\documentclass[11pt, a4paper]{article}

\usepackage{latexsym}
\usepackage{amsmath, epsfig}
\usepackage{color}
\usepackage{graphicx, float}
\usepackage{amssymb,amsfonts,amsthm,amsmath,graphicx,url}
\usepackage{authblk}
\usepackage{enumerate}

\definecolor{VeryLightBlue}{rgb}{0.9,0.9,1}
\definecolor{LightBlue}{rgb}{0.8,0.8,1}
\definecolor{MidBlue}{rgb}{0.5,0.5,1}
\definecolor{DarkBlue}{rgb}{0,0,0.6}
\definecolor{Blue}{rgb}{0,0,1}
\definecolor{Gold}{rgb}{1,0.843,0}
\definecolor{LightGreen}{rgb}{0.88,1,0.88}
\definecolor{MidGreen}{rgb}{0.6,1,0.6}
\definecolor{DarkGreen}{rgb}{0,0.6,0}
\definecolor{VeryLightYellow}{rgb}{1,1,0.9}
\definecolor{LightYellow}{rgb}{1,1,0.6}
\definecolor{MidYellow}{rgb}{1,1,0.5}
\definecolor{DarkYellow}{rgb}{1,1,0.2}
\definecolor{DarkPurple}{rgb}{.6,0,1}
\definecolor{Red}{rgb}{1,0,0}
\definecolor{VeryLightRed}{rgb}{1,0.9,0.9}
\definecolor{LightRed}{rgb}{1,0.8,0.8}
\definecolor{MidRed}{rgb}{1,0.55,0.55}

\long\def\delete#1{}

\usepackage[normalem]{ulem}

\newtheorem{theorem}{Theorem}[section]
\newtheorem{lemma}[theorem]{Lemma}
\newtheorem{corollary}[theorem]{Corollary}
\newtheorem{proposition}[theorem]{Proposition}

\theoremstyle{definition}

\newtheorem{remark}{Remark}[section]

\input amssym.def
\input amssym.tex

\newcommand{\bmat}[1]{\begin{bmatrix}#1\end{bmatrix}}

\newcommand{\be}{\begin{equation}}
\newcommand{\ee}{\end{equation}}
\newcommand{\bea}{\begin{eqnarray}}
\newcommand{\eea}{\end{eqnarray}}
\newcommand{\bean}{\begin{eqnarray*}}
\newcommand{\eean}{\end{eqnarray*}}

\def\qed{\hfill$\Box$\vspace{11pt}}

\def\la{\langle}
\def\ra{\rangle}

\def\Id{{\mathrm{Id}}}
\def\id{{\mathrm{id}}}

\def\CCC{\mathbb{C}}

\def\ZZZ{\mathbb{Z}}

\def\RRR{\mathbb{R}}

\def\QQ{{\cal Q}}

\def\b0{{\bf 0}}

\def\Ga{\Gamma}

\def\a{\alpha}
\def\b{\beta}
\def\de{\delta}
\def\d{\mathbf{1}}

\def\l{\lambda}

\def\vp{\varphi}

\def\Sym{{\rm Sym}}
\def\Inn{{\rm Inn}}

\def\Cay{{\rm Cay}}
\def\Aut{{\rm Aut}}

\def\SL{{\rm SL}}

\def\det{{\rm det}}

\def\Ker{{\rm Ker}}

\def\Cay{{\rm Cay}}
\def\rank{\mathrm{rank}}

\title{Equitable partitions of regular graphs, and perfect sets in normal Cayley graphs}

\author[1]{R. A. Bailey}
\author[1]{Peter J. Cameron}
\author[2]{Sanming Zhou}

\affil[1]{\small School of Mathematics and Statistics, University of St Andrews, St Andrews, Fife KY16 9SS, UK}
\affil[2]{\small School of Mathematics and Statistics, The University of Melbourne, Parkville, VIC 3010, Australia}

\date{}

\begin{document}
\openup 0.5\jot
\maketitle

\renewcommand{\thefootnote}{\empty}%{footnote}}
\footnotetext{E-mail addresses: rab24@st-andrews.ac.uk (R. A. Bailey), pjc20@st-andrews.ac.uk (P. J. Cameron), sanming@unimelb.edu.au (S. Zhou)}

\begin{abstract}
An equitable partition of a graph $\Ga$ is a partition $\{V_1, \ldots, V_m\}$ of its vertex set such that for each pair $i, j$ all vertices in $V_i$ have the same number of neighbours in $V_j$. When $m=2$, $V_1$ is called an $(a, b)$-perfect set in $\Ga$, where $a$ is the number of neighbours in $V_1$ of each vertex in $V_1$, and $b$ is the number of neighbours in $V_1$ of each vertex in $V_2$. In this paper we first derive general necessary conditions for a regular graph to admit two equitable partitions. As a corollary we obtain necessary conditions for the existence of an $(a,b)$-perfect set  in a regular graph in terms of an arbitrary equitable partition. With the help of these results we then obtain necessary conditions for the existence of an $(a,b)$-perfect set in a normal Cayley graph in terms of the irreducible characters of the underlying group.

{\em Key words}: Equitable partition, perfect set, perfect code, Cayley graph

{\em AMS Subject Classification (2020)}: 05C25, 05C69, 94B99
\end{abstract}

\section{Introduction}
\label{sec:intro}

All groups considered in this paper are finite, and all graphs considered are finite and undirected with no loops or parallel edges. Let $\Ga = (V, E)$ be a graph. A partition 
$$
\pi=\{V_1, \ldots, V_m\}
$$
of $V$ is called an \emph{equitable partition} \cite{GR} of $\Ga$ if there exists an $m \times m$ matrix 
$$
M_{\pi} = (b_{ij})_{1 \le i, j \le m},
$$
called the \emph{quotient matrix} of $\pi$, such that for any $i, j$ with $1 \le i, j \le m$, every vertex in $V_i$ has exactly $b_{ij}$ neighbours in $V_j$. In the literature an equitable partition is also called a \emph{perfect colouring} \cite{F2007} or, more specifically, \emph{perfect $m$-colouring} to specify the number of colour classes. In particular, if $\{V_1, V_2\}$ is a perfect $2$-colouring of $\Ga$ with quotient matrix $(b_{ij})_{1 \le i, j \le 2}$, then $V_1$ is called a \emph{$(b_{11}, b_{21})$-perfect set} \cite{BCGG2019} or \emph{$(b_{11}, b_{21})$-regular set} \cite{Cardoso2019, WXZ23, WXZ24} in $\Ga$ (so $V_2$ is a $(b_{22}, b_{12})$-perfect set or $(b_{22}, b_{12})$-regular set in $\Ga$). Thus, if $\Ga$ is $k$-regular and connected, and $a$ and $b$ are integers with $0 \le a \le k-1$ and $1 \le b \le k$, then a nonempty proper subset $V_1$ of $V$ is an $(a,b)$-perfect set in $\Gamma$ if and only if $\{V_1,V_2\}$ is an equitable partition of $\Ga$ with quotient matrix
\be
\label{eq:qtabd}
\bmat{a & k-a\\b & k-b},
\ee
where $V_2 = V \setminus V_1$. It is clear that $|V_1| \ge b$, $|V_2| \ge k-a$,  and
\be
\label{eq:abd}
(k-a) |V_1| = b |V_2|. 
\ee
Thus, if $\Ga$ admits an $(a, b)$-perfect set, then $|V| \ge k-a+b$ and moreover $k-a+b$ is a common divisor of $(k-a)|V|$ and $b|V|$ (as \eqref{eq:abd} can be written as $(k-a+b)|V_1| = b|V|$ or $(k-a+b) |V_2| = (k-a)|V|$). In the literature a $(0, 1)$-perfect set is called a \emph{perfect $1$-code} \cite{Big, G1993}, an \emph{efficient dominating set} \cite{DS03} or \emph{independent perfect dominating set} \cite{Lee01}, and a $(1, 1)$-perfect set is called a \emph{total perfect code} \cite{Zhou2016} or an \emph{efficient open dominating set} \cite{HHS98}. In this paper, we use the term \emph{perfect code} to mean a perfect 1-code.

Historically, equitable partitions are associated with distance-regular graphs and completely regular codes (see, for example, \cite{G1993,GR,K11,KP25,M03}). They are also related to experimental designs with commutative
orthogonal block structure \cite{BCFFN24}. Assume that $\Ga$ is connected. If $u \in V$, let $\Ga_{i}(u)$ denote the set of vertices of $\Ga$ at distance $i$ from $u$. Obviously, $\Ga_{0}(u) = \{u\}$ and $\Ga_{1}(u)$ (denoted by $\Ga(u)$) is the neighbourhood of $u$. The \emph{distance partition} of $\Ga$ with respect to $u$, denoted by $\pi_d(u)$, is the partition $\{\Ga_{i}(u): i \ge 0\}$ of $V$. In general, for $C \subseteq V$, define $C_i$ to be the set of vertices of $\Ga$ at distance $i$ from $C$, where the distance from a vertex $v \in V$ to $C$ is the minimum distance between $v$ and a vertex in $C$. If $\rho$ is the \emph{covering radius} of $C$, namely the least integer $i$ such that every vertex in $\Ga$ is at distance no more than $i$ to at least one vertex in $C$, then the partition $\pi_d(C) = \{C_0, C_1, \ldots, C_{\rho}\}$ (where $C_0 = C$) of $V$ is called the \emph{distance partition} of $C$. (In particular, $\pi_d(\{u\})$ is $\pi_d(u)$ for each $u \in V$.) Assume further that $\Ga$ is $k$-regular with diameter $d$. If $\pi_d(C)$ is an equitable partition of $\Ga$, then $C$ is said to be a \emph{completely regular code} \cite{Neu92} in $\Ga$. If $C$ is a completely regular code, then $M_{\pi_d(C)}$ is a tridiagonal matrix. A proper subset $C$ of $V$ is a completely regular code in $\Ga$ with covering radius $1$ if and only if it is an $(a, b)$-regular set in $\Ga$ for some $0 \le a \le k-1$ and $1 \le b \le k$. It is widely known \cite{G1993, Neu92} that $\Ga$ is distance-regular if and only if $\pi_d(u)$ is an equitable partition for every vertex $u \in V$. If $\Ga$ is distance-regular with intersection array
$$
\bmat{b_0 & b_1 & \ldots & b_{d-1} & b_d = 0 \\
c_0 = 0 & c_1 & \ldots & c_{d-1} & c_d},
$$
then $M_{\pi_d(u)}$ is the tridiagonal matrix with main diagonal entries $a_0, a_1, \ldots, a_{d-1}, a_d$, upper diagonal entries $b_0, b_1, \ldots, b_{d-1}$, and lower diagonal entries $c_1, \ldots, c_{d-1}, c_d$, where $b_0 = k, c_1 = 1$ and $a_i = k - b_i - c_i$ for $0 \le i \le d$. 

It is known that, for any equitable partition $\pi$ of a graph $\Ga = (V, E)$, all eigenvalues of $M_{\pi}$ are eigenvalues of $\Gamma$ (see \cite[Theorem 9.3.3]{GR}). Now assume that $\Gamma$ is connected and $k$-regular. Then $k$ is a simple eigenvalue of $M_{\pi}$, and the equitable partition $\pi$ of $\Gamma$ is said to be \emph{$\mu$-equitable} \cite{BCGG2019} if all eigenvalues of $M_{\pi}$ other than $k$ are equal to $\mu$. In particular, if an equitable partition $\{V_1, V_2\}$ into two parts is $\mu$-equitable, then the nonempty proper subset $V_1$ of $V$ is called a \emph{$\mu$-perfect set} \cite{BCGG2019} in $\Ga$. It was proved in \cite[Proposition 2.1]{BCGG2019} that, for a partition $\pi = \{V_1, V_2, \dots, V_m\}$ of $V$, if $\pi$ is $\mu$-equitable, then each $V_i$ is $\mu$-perfect, and conversely if $V_1, V_2, \dots, V_{m-1}$ are all $\mu$-perfect, then $\pi$ is $\mu$-equitable. Thus, it is especially important to study equitable partitions with exactly two parts (which are essentially perfect sets). Since the eigenvalues of the matrix in \eqref{eq:qtabd} are $k$ and $a-b$, if $\Ga$ admits an $(a, b)$-perfect set, then $a-b$ is an eigenvalue of $\Ga$. In \cite{GG2013}, perfect $2$-colourings of Johnson graphs $J(v,3)$ were classified for odd $v$, and some recent results on perfect $2$-colourings of Hamming graphs can be found in \cite{BKMV2021, MV2020}. In \cite{BCGG2019}, equitable partitions of Latin square graphs are studied and those whose quotient matrix does not have an eigenvalue $-3$ are classified. 

In recent years, perfect sets, and in particular perfect codes, in Cayley graphs have received considerable attention. Given a group $G$ with identity element $1$ and an inverse-closed subset $S$ of $G \setminus \{1\}$, the \emph{Cayley graph} $\Cay(G,S)$ of $G$ with respect to the \emph{connection set} $S$ is the graph with vertex set $G$ such that $x, y \in G$ are adjacent if and only if $yx^{-1}\in S$. Obviously, $\Cay(G,S)$ is an undirected, simple and $|S|$-regular graph, and it is connected if and only if $\la S \ra = G$; and $G$ acts by right
multiplication as a group of automorphisms of $\Cay(G,S)$.

A Cayley graph $\Cay(G,S)$ is called \emph{normal} if $S$ is closed under conjugation, or equivalently, $S$ is the union of some conjugacy classes of $G$. (Note that this definition is one of the two widely used notions of the normality of a Cayley graph. The other one says that a Cayley graph $\Cay(G,S)$ is normal if the right regular representation of $G$ is normal in the full automorphism group of $\Cay(G,S)$. These definitions are quite different, but in this paper we use the one given at the start of this paragraph.) If $\Cay(G,S)$ is a normal Cayley graph, then both the left and right actions of $G$ are automorphisms.

The Hamming graph $H(n, q)$ and the Cartesian product $L(n, q)$ of $n$ copies of $C_q$ (the cycle of length $q$) are both Cayley graphs of $\ZZZ_q^n$, and the Hamming and Lee metrics over $q$-ary words of length $n$ are exactly the graph distances in $H(n, q)$ and $L(n, q)$, respectively. Thus, perfect codes in $H(n, q)$ and $L(n, q)$ are precisely perfect $1$-codes under the Hamming and Lee metrics, respectively. So perfect codes in Cayley graphs can be considered as a generalisation of perfect codes \cite{Heden1, Va75} in coding theory. Moreover, perfect codes in Cayley graphs are essentially factorisations \cite{SS2009} of the underlying groups. In fact, for subsets $X, Y$ of $G$ with $1 \in X \cap Y$, if $X$ is inverse-closed, then $Y$ is a perfect code in $\Cay(G, X \setminus \{1\})$ if and only if $G = XY$ is a factorisation of $G$ (that is, every element of $G$ can be written uniquely as $xy$ with $x \in X$ and $y \in Y$). Because of this, perfect codes in Cayley graphs have attracted considerable attention over the years; see \cite[Section 1]{HXZ18} for a brief account of results and \cite{CWX2020, FHZ, HXZ18, MWWZ2019, ZZ2021, Zhou2016, Z15} for several recent papers. As early as 1966, Rothaus and Thompson \cite{RT1966} proved (in the language of the present paper) that the complete transposition graph on $S_n$ (that is, the Cayley graph of the symmetric group $S_n$ with connection set consisting of all transpositions in $S_n$) does not admit any perfect code if $1 + (n(n-1)/2)$ is divisible by a prime exceeding $2 + \sqrt{n}$. The reader is referred to \cite{WXZ23, WXZ24} for two recent studies of perfect sets in Cayley graphs. 

In this paper we study equitable partitions of regular graphs with an emphasis on perfect sets in normal Cayley graphs. In the next section we give necessary conditions for a regular graph to admit two equitable partitions (see Theorem \ref{thm:2equ}). This general result can be applied to various pairs of (not necessarily distinct) equitable partitions such as the distance partition of a distance-regular graph, the distance partition of a completely regular set, etc. As a special case we obtain necessary conditions for the existence of an $(a,b)$-perfect set in terms of an arbitrary equitable partition in a regular graph (see Corollary \ref{coro:2equ-ps}). In another special case we obtain linear-algebraic and spectral conditions for the existence of an equitable partition of a regular graph (see Corollary \ref{coro:2equ-tr}). With the help of these results we obtain in Section \ref{sec:eqCay} two sets of necessary conditions for the existence of an $(a,b)$-perfect set in a normal Cayley graph in terms of the irreducible characters of the underlying group (see Theorems \ref{thm:nec-cond} and \ref{thm:nec-cd}). In the special case when $(a, b) = (0,1)$ or $(1,1)$, these results recover two results in \cite{E87} and two results in \cite{Zhou2016}, respectively. We conclude the paper in Section \ref{sec:rem} by giving lower bounds on the multiplicity of the eigenvalue $a-b$ of a vertex-transitive graph admitting an $(a,b)$-perfect set (see Theorem \ref{thm:vt}) and determining all perfect sets in a particular family of normal Cayley graphs of dihedral groups (see Proposition~\ref{ex:dih}).

\section{Equitable partitions of regular graphs}
\label{sec:eqReg}

In this section we study equitable partitions of regular graphs. A special case of the main result in this section will be used in our study of perfect sets in normal Cayley graphs in the next section. 

Let $\Ga = (V, E)$ be a $k$-regular graph, where $k \ge 1$. As usual, for a subset $X$ of $V$, we use $\Ga[X]$ to denote the subgraph of $\Ga$ induced by $X$. Denote by $\CCC V$ the vector space of functions from $V$ to $\CCC$, with usual addition and scalar multiplication. We may identify each $f \in \CCC V$ with the column vector $[f(v)]_{v \in V}$. Denote by $\d_X \in \CCC V$ the characteristic function of a subset $X$ of $V$. That is, $\d_{X}(v) = 1$ if $v \in X$ and $\d_{X}(v) = 0$ if $v \in V \setminus X$. In particular, $\d_V(v) = 1$ for all $v \in V$ and $\d_V$ is the all-one column vector. For each $v \in V$, write $\d_v$ in place of $\d_{\{v\}}$. Then $\{\d_v: v \in V\}$ is a basis of $\CCC V$ as it is independent and $f = \sum_{v \in V} f(v) \d_v$ for any $f \in \CCC V$. For a partition $\pi = \{V_1, \ldots, V_m\}$ of $V$, define $\CCC \pi$ to be the subspace of $\CCC V$ spanned by $\d_{V_1}, \ldots, \d_{V_m}$. Note that $\dim(\CCC \pi) = m$. Note also that a function in $\CCC V$ which is constant on each part of $\pi$ induces a function in $\CCC \pi$ in a natural way.

Let
$$
A = (a_{uv})_{u, v \in V}
$$ 
be the adjacency matrix of $\Ga$. Define the adjacency linear mapping \cite{E87} 
$$
\varphi: \CCC V \rightarrow \CCC V
$$ 
by
\be
\label{eq:vphi}
\vp(\d_v) := A \d_v = \sum_{u \in V} a_{uv} \d_{u} = \sum_{u \in \Ga(v)} \d_{u}\; \text{for } v \in V.
\ee
Then the matrix of $\vp$ with respect to the basis $\{\d_v: v \in V\}$ of $\CCC V$ is $A$, and $\vp$ maps each column vector $f \in \CCC V$ to $Af$. Denote by $\id$ the identity map from $\CCC V$ to $\CCC V$. 

\begin{lemma}
[{\cite[Lemma 2(ii)]{E87}, \cite[Lemma 9.3.2]{GR}}] 
\label{lem:equi-part}
Let $\Ga = (V, E)$ be a $k$-regular graph and $\pi = \{V_1, \ldots, V_m\}$ a partition of $V$. Then $\pi$ is an equitable partition of $\Ga$ if and only if $\CCC \pi$ is invariant under the linear mapping $\vp$. 
\end{lemma}

In what follows, we assume that $\pi = \{V_1, \ldots, V_m\}$ is an equitable partition of $\Ga$ with quotient matrix $M_{\pi} = (b_{ij})_{1 \le i,j \le m}$. Then $b_{ij} = |\Ga(v) \cap V_j|$ for all $v \in V_i$. Hence the entries in each row of $M_\pi$ add up to $k$. Denote by $\vp_{\pi}$ and $\Id_{\pi}$ the restriction of $\vp$ and $\id$ to $\CCC \pi$, respectively. Then the matrix of $\vp_{\pi}$ with respect to the basis $\{\d_{V_1}, \ldots, \d_{V_m}\}$ of $\CCC \pi$ is exactly $M_{\pi}$. In fact, for $1 \le j \le m$, we have
\begin{eqnarray*}
\vp_{\pi}(\d_{V_j})
&=& \sum_{v \in V_j} \vp(\d_{v})
= \sum_{v \in V_j} \sum_{u \in V} a_{uv} \d_{u}\\
&=& \sum_{v \in V_j} \sum_{i=1}^m \sum_{u \in V_i} a_{uv} \d_{u}
= \sum_{i=1}^m \sum_{u \in V_i} \sum_{v \in V_j} a_{uv} \d_{u}\\
&=& \sum_{i=1}^m \sum_{u \in V_i} b_{ij} \d_{u}
= \sum_{i=1}^m b_{ij} \d_{V_i}. 
\end{eqnarray*}

From now on, for $v$ in $V$, we use the notation $\pi[v]$ to denote the part of $\pi$ containing $v$.  Define
$$
p: \CCC V \rightarrow \CCC \pi
$$
by 
$$
p(f) = \sum_{i=1}^{m} \left(\frac{\sum_{v \in V_i}f(v)}{|V_i|}\right) \d_{V_i}\; \text{for } f \in \CCC V.
$$
Obviously, $p$ is a linear mapping from $\CCC V$ to $\CCC \pi$. Let $P$ be the matrix of $p$ with respect to the basis $\{\d_v: v \in V\}$ of $\CCC V$. Then
$$
p(\d_{v}) = \sum_{i=1}^{m} \left(\frac{\sum_{w \in V_i}\d_{v}(w)}{|V_i|}\right) \d_{V_i} = \frac{1}{|\pi[v]|} \d_{\pi[v]} = \sum_{u \in \pi[v]}\frac{1}{|\pi[u]|} \d_{u} = \sum_{u \in V}\frac{\d_{\pi[u]}(v)}{|\pi[u]|} \d_{u}.
$$
Thus, the $(u, v)$-entry of $P$ is given by 
$$
P_{uv} = \frac{\d_{\pi[u]}(v)}{|\pi[u]|}.
$$  
Using this and the assumption that $\pi$ is equitable, one can verify (see, for example, \cite{E87}) that $PA = AP$. In other words, 
\be
\label{eq:pvp}
p \vp = \vp p.
\ee    

\begin{lemma}
\label{lem:ps}
Let $\Ga = (V, E)$ be a $k$-regular graph, where $k \ge 1$. A partition $\pi = \{V_1, \ldots, V_m\}$ of $V$ is an equitable partition of $\Ga$ with quotient matrix $M_{\pi} = (b_{ij})_{1 \le i, j \le m}$ if and only if  
\begin{equation}
\label{eq:ns0}
\vp(\d_{V_j}) = \sum_{i=1}^m b_{ij}\d_{V_i} \hbox{ for }1 \le j \le m. 
\end{equation}
\end{lemma}

\begin{proof}
For each $1 \le j \le m$ and any $u \in V$, by \eqref{eq:vphi} we have 
\begin{eqnarray*}
\vp(\d_{V_j})(u) &=& \left(\sum_{v \in V_j} \vp(\d_v)\right)(u) = \sum_{v \in V_j} \vp(\d_v)(u)\\
&=& \sum_{v \in V_j} \left(\sum_{w \in \Ga(v)} \d_{w}\right)(u) = \sum_{v \in V_j}a_{uv} = |\Ga(u) \cap V_j|.
\end{eqnarray*}
Also, if $u \in V_i$, then 
\[\left(\sum_{l=1}^m b_{lj}\d_{V_l}\right)(u) = \sum_{l=1}^m b_{lj}\d_{V_l}(u) = b_{ij}.\] The result follows immediately from these equations and the definition of an equitable partition. 
\qed
\end{proof}

In the case when $m = 2$, Lemma \ref{lem:ps} gives the trivial result that a subset $V_1\subseteq V$ is an $(a, b)$-perfect set in $\Ga$ if and only if $\vp(\d_{V_1}) = a \d_{V_1} + b \d_{V_2}$ (or, equivalently, $\vp(\d_{V_2}) = (k-a) \d_{V_1} + (k-b) \d_{V_2}$), where $V_2=V\setminus V_1$. 

In the following we assume that
$$
\tau = \{W_1, \ldots, W_n\}
$$
is a second equitable partition of $\Ga$ with quotient matrix
$$
M_{\tau} = (c_{ij})_{1 \le i, j \le n}.
$$
For $1 \le i, j \le n$, we have $k+c_{ij}-c_{jj} \ge 0$ as $k-c_{jj} \ge 0$ and $c_{ij} \ge 0$. Moreover, $k+c_{ij}-c_{jj} = 0$ if any only if $i \ne j$ and the subgraph $\Ga[W_j]$ of $\Ga$ induced by $W_j$ is a connected component of $\Ga$. In all other cases, we will use the notation
$$
r_{ij}=\frac{c_{ij}}{k+c_{ij}-c_{jj}}\;\, \text{for } 1\le i,j\le n.
$$
As usual, denote by $\de$ the Kronecker delta.

The main result in this section is as follows. 

\begin{theorem}
\label{thm:2equ}
Let $\Ga = (V, E)$ be a $k$-regular graph, where $k \ge 1$. Let $\pi = \{V_1, \ldots, V_m\}$ and $\tau = \{W_1, \ldots, W_n\}$ be (not necessarily distinct) equitable partitions of $\Ga$, with quotient matrices $M_{\pi} = (b_{ij})_{1 \le i,j \le m}$ and $M_{\tau} = (c_{ij})_{1 \le i, j \le n}$, respectively. Then, for any $1 \le i, j \le n$ with $i \ne j$ such that $\Ga[W_j]$ is not a connected component of $\Ga$, the following hold:
\begin{enumerate}[\rm (a)]
\item
the system of linear equations
\be
\label{eq:2equ}
(M_{\tau} + (c_{ij} - c_{jj}) I_n) \bmat{x_1\\ \vdots \\ x_n} = \bmat{(c_{1j} - c_{ij}) + \de_{1j}(c_{ij} - c_{jj}) \\ \vdots \\ (c_{nj} - c_{ij}) + \de_{nj}(c_{ij} - c_{jj})}
\ee
is consistent; 
\item 
for any solution $(d_1, \ldots, d_n)$ to \eqref{eq:2equ}, the vector $[h_1, \ldots, h_m]^\top$ defined by
\be
\label{eq:kt}
h_t = - r_{ij} - \sum_{l = 1}^n \frac{|V_t \cap W_l|}{|V_t|}(d_l - \de_{lj})\;\, \text{for } 1 \le t \le m
\ee
is a solution to the homogeneous system of linear equations
\be
\label{eq:2equhom}
(M_{\pi} + (c_{ij} - c_{jj}) I_m) \bmat{y_1\\ \vdots \\ y_m} = \bmat{0 \\ \vdots \\ 0};
\ee
\item
in particular, either $(h_1, \ldots, h_m) = (0, \ldots, 0)$ or $c_{jj} - c_{ij}$ is an eigenvalue of $M_{\pi}$ (and hence an eigenvalue of $\Ga$) with $[h_1, \ldots, h_m]^\top$ as a corresponding eigenvector. 
\end{enumerate}
\end{theorem}

\begin{proof}
Fix $i \ne j$ with $1 \le i, j \le n$ such that $\Ga[W_j]$ is not a connected component of $\Ga$. Then $k+c_{ij}-c_{jj} > 0$.

\medskip

(a) Substituting $(\de_{1j}-r_{ij}, \ldots, \de_{nj}-r_{ij})$ for $(x_1, \ldots, x_n)$ in \eqref{eq:2equ}, the $l$-th coordinate of the left-hand side of \eqref{eq:2equ} becomes 
\begin{eqnarray*}
-\left(\sum_{s=1}^{n} r_{ij}c_{ls}\right) + c_{lj} + (\de_{lj}-r_{ij})(c_{ij} - c_{jj}) 
&=& -r_{ij}k + c_{lj} + (\de_{lj}-r_{ij})(c_{ij} - c_{jj})\\
&=& -r_{ij}(k + c_{ij} - c_{jj}) + c_{lj} + \de_{lj}(c_{ij} - c_{jj})\\
&=&  (c_{lj} - c_{ij}) + \de_{lj}(c_{ij} - c_{jj}),
\end{eqnarray*}
which agrees with the right-hand side of \eqref{eq:2equ}. Hence $(\de_{1j}-r_{ij}, \ldots, \de_{nj}-r_{ij})$ is a solution to \eqref{eq:2equ}. Therefore, \eqref{eq:2equ} is consistent. 

\medskip

(b) Now suppose that $(d_1, \ldots, d_n)$ is an arbitrary solution to \eqref{eq:2equ}. Let $(h_1, \ldots, h_m)$ be defined by \eqref{eq:kt}. By \eqref{eq:pvp} and \eqref{eq:ns0}, 
\begin{eqnarray*}
\vp p(\d_{W_j}) & = & p \vp(\d_{W_j}) \\
& = & \sum_{l=1}^n c_{lj}p(\d_{W_l}) \\
& = & c_{ij}p(\d_{W_i}) + \sum_{l \ne i} c_{lj}p(\d_{W_l}) \\
& = & c_{ij}\left(p(\d_{V}) - \sum_{l \ne i}p(\d_{W_l})\right) + \sum_{l \ne i} c_{lj}p(\d_{W_l}) \\
& = & c_{ij} p(\d_{V}) + \sum_{l \ne i} (c_{lj} - c_{ij})p(\d_{W_l}). 
\end{eqnarray*}
That is, 
\be
\label{eq:Id}
\left(\vp + (c_{ij} - c_{jj})\Id\right)(p(\d_{W_j})) = c_{ij} p(\d_{V}) + \sum_{l \ne i, j} (c_{lj} - c_{ij})p(\d_{W_l}),
\ee
where $\Id$ is the identity transformation.
On the other hand, by \eqref{eq:ns0}, we have
\begin{eqnarray*}
\vp\left(r_{ij}\d_{V} + \sum_{l=1}^n d_{l} \d_{W_l}\right) & = & r_{ij}\vp(\d_{V}) + \sum_{t=1}^n d_{t} \vp(\d_{W_t}) \\
& = & r_{ij} k\d_{V} + \sum_{t=1}^n d_{t} \left(\sum_{l=1}^n c_{lt}\d_{W_l}\right) \\
& = & r_{ij} k\d_{V} + \sum_{l=1}^n  \left(\sum_{t=1}^n c_{lt}d_{t}\right)\d_{W_l}.
\end{eqnarray*}
Hence, by \eqref{eq:pvp} and \eqref{eq:ns0}, the fact that $r_{ij}=c_{ij}/(k+c_{ij}-c_{jj})$, and then the assumption that $(d_1, \ldots, d_n)$ is a solution to \eqref{eq:2equ}, we have
\begin{eqnarray*}
\left(\vp + (c_{ij} - c_{jj})\Id\right)p\left(r_{ij}\d_{V} + \sum_{l=1}^n d_{l} \d_{W_l}\right)& = & (c_{ij} - c_{jj}) \left(r_{ij} p(\d_{V}) + \sum_{l=1}^n d_{l} p(\d_{W_l})\right) + \\
& & r_{ij} k p(\d_{V}) + \sum_{l=1}^n  \left(\sum_{t=1}^n c_{lt}d_{t}\right)p(\d_{W_l}) \\
& = & c_{ij}p(\d_{V}) + \sum_{l=1}^n  \left((c_{ij} - c_{jj})d_{l} + \sum_{t=1}^n c_{lt}d_{t}\right)p(\d_{W_l}) \\ 
& = & c_{ij}p(\d_{V}) + \sum_{l=1}^n  \left((c_{lj} - c_{ij}) + \de_{lj}(c_{ij} - c_{jj})\right)p(\d_{W_l}) \\
& = & c_{ij} p(\d_{V}) + \sum_{l \ne i, j} (c_{lj} - c_{ij})p(\d_{W_l}).
\end{eqnarray*}

Combining this with \eqref{eq:Id}, we obtain 
\be
\label{eq:pdw}
p(\d_{W_j}) - r_{ij}p(\d_{V}) - \sum_{l=1}^n d_{l} p(\d_{W_l}) \in \Ker\left(\vp_{\pi} + (c_{ij} - c_{jj})\Id_{\pi}\right).
\ee

By the definition of $p$, we have $p(\d_{V}) = \d_{V} = \sum_{t=1}^m \d_{V_t}$ and 
$$
p(\d_{W_l}) = \sum_{t=1}^{m} \left(\frac{\sum_{v \in V_t}\d_{W_l}(v)}{|V_t|}\right) \d_{V_t} = \sum_{t=1}^{m} \frac{|V_t \cap W_l|}{|V_t|} \d_{V_t}
$$ 
for $1 \le l \le n$. So we have 
\begin{eqnarray*}
p(\d_{W_j}) - r_{ij}p(\d_{V}) - \sum_{l=1}^n d_{l} p(\d_{W_l}) & = & \sum_{t=1}^{m} \frac{|V_t \cap W_j|}{|V_t|} \d_{V_t} - r_{ij}\sum_{t=1}^m \d_{V_t} - \sum_{l=1}^n d_{l} \left(\sum_{t=1}^{m} \frac{|V_t \cap W_l|}{|V_t|} \d_{V_t}\right) \\
& = & \sum_{t=1}^{m} \left(\frac{|V_t \cap W_j|}{|V_t|}-r_{ij}-\sum_{l=1}^n d_{l}\frac{|V_t \cap W_l|}{|V_t|} \right)\d_{V_t} \\
& = & \sum_{t=1}^{m} h_t\d_{V_t},
\end{eqnarray*}	
where $h_t$ is as defined in \eqref{eq:kt}. Since the matrices of $\vp_{\pi}$, $\Id_{\pi}$ with respect to the basis $\{\d_{V_1}, \ldots, \d_{V_m}\}$ of $\CCC \pi$ are $M_{\pi}$, $I_m$, respectively, it follows from \eqref{eq:pdw} that the vector $[h_1, \ldots, h_m]^\top$ is a solution to \eqref{eq:2equhom}.

\medskip

(c) This statement now follows.
\qed
\end{proof}

\begin{remark}
\label{rem:2equ} 
(a) If $\Ga$ is connected and $n \ge 2$, then for any $1 \le j \le n$, $\Ga[W_j]$ is not a connected component of $\Ga$, and Theorem \ref{thm:2equ} holds for any $1 \le i, j \le n$ with $i \ne j$.  

(b) The matrix $\vp_\pi$ in \eqref{eq:pdw} is exactly the matrix $M_\pi$. Therefore \eqref{eq:pdw} is equivalent to $(h_1, \ldots, h_m)$ being a solution to \eqref{eq:2equhom}. So the second statement in Theorem \ref{thm:2equ} can be expressed as \eqref{eq:pdw}. 

(c) For $\ell=1, \ldots,n$, put $d_\ell^*=\delta_{lj}-r_{ij}$. The proof of  
part (a) shows that $(d_1^*, \ldots, d_n^*)$ is a solution to
(\ref{eq:2equ}). Equation~(\ref{eq:kt}) shows that the corresponding
$(h_1,\ldots,h_m)$ is $(h_1^*,\ldots,h_m^*)=(0,\ldots,0)$. The general
solution to (\ref{eq:2equ}) has the form
$(a_1,\ldots,a_n)+(d_1^*,\ldots,d_n^*)$, where either
$(a_1,\ldots,a_n)=(0,\ldots,0)$, or $[a_1,\ldots,a_n]^\top$ is an eigenvector of
$M_\tau$ with eigenvalue $c_{jj}-c_{ij}$. In the second case,
(\ref{eq:kt}) shows that the corresponding $(h_1, \ldots, h_m)$ is given
by $\displaystyle{h_t = - \sum_{l = 1}^n \frac{|V_t \cap W_l|}{|V_t|}a_l}$ for $1 \le t \le m$.
In the special case when $\pi = \tau$, we can label the parts so that
$V_t = W_t$ for $1 \le t \le m=n$, and then $h_t = - a_t$ for each $t$; and so
the second and third statements in Theorem \ref{thm:2equ} are trivially true.
\end{remark}

\begin{remark}
(a) Theorem \ref{thm:2equ} can be applied to various pairs of equitable partitions $(\pi, \tau)$ of a regular graph $\Ga$. In particular, it gives necessary conditions for the existence of a completely regular code $C$ in $\Ga$ if we choose $\pi$ or $\tau$ to be $\pi_d(C)$ (so $M_{\pi}$ or $M_{\tau}$ is tridiagonal, rendering a simpler computation in \eqref{eq:2equ}--\eqref{eq:2equhom}). In Corollary \ref{coro:2equ-ps} below, we consider the case when $C$ has covering radius $1$ (that is, $C$ is a perfect set).  

(b) If $\Ga$ is distance-regular, then in Theorem \ref{thm:2equ} we can choose $\pi$ or $\tau$ to be $\pi_d(u)$ for some $u \in V$ (so $M_{\pi}$ or $M_{\tau}$ is tridiagonal). For example, in Corollary \ref{coro:2equ-ps} below, one of $\pi$ and $\tau$ is set to be $\pi_d(C)$ for a perfect set $C$, and the other can be chosen to be $\pi_d(u)$ when $\Ga$ is distance-regular.  

(c) In Corollary \ref{coro:2equ-tr} below, we can choose $\tau$ to be $\pi_d(C)$ for a completely regular code $C$ in $\Ga$, thereby obtaining necessary conditions for the existence of such a code. If $\Ga$ is distance-regular, we can also choose $\tau$ to be $\pi_d(u)$ for some $u \in V$. 

(d) If $G$ is a subgroup of the automorphism group $\Aut(\Ga)$ of $\Ga$, then the $G$-orbits on $V$ form an equitable partition of $\Ga$. In Theorem \ref{thm:2equ}, we may choose $\pi$ or $\tau$ to be such a partition for some $G \le \Aut(\Ga)$. 
\end{remark}

\begin{corollary}
\label{coro:2equ-ps}
Let $\Ga = (V, E)$ be a connected $k$-regular graph, where $k \ge 1$. Let $\pi = \{V_1, \ldots, V_m\}$ be an equitable partition of $\Ga$ with quotient matrix $M_{\pi} = (b_{ij})_{1 \le i,j \le m}$ and $W_1$ an $(a, b)$-perfect set in $\Ga$, where $0 \le a \le k-1$ and $1 \le b \le k$. 
\begin{enumerate}[\rm (a)]
\item 
The vector $[h_1, \ldots, h_m]^\top$ defined by
\be
\label{eq:eigen}
h_t = \frac{|V_t \cap W_1|}{|V_t|} - \frac{b}{k-a+b}\; \text{ for } 1 \le t \le m
\ee  
is a solution to the homogeneous system of linear equations
\be
\label{eq:eigen1}
(M_{\pi} + (b-a)I_m) \bmat{y_1 \\ \vdots \\ y_m} = \bmat{0 \\ \vdots \\ 0}.
\ee
In particular, either 
\[\frac{|V_t \cap W_1|}{|V_t|} = \frac{b}{k-a+b}\]
for each $1 \le t \le m$, or $a-b$ is an eigenvalue of $M_{\pi}$ (and hence an eigenvalue of $\Ga$) with $[h_1, \ldots, h_m]^\top$ as a corresponding eigenvector.
\item 
For any $1 \le i, j \le m$ with $i \ne j$, the following hold:
\begin{enumerate}[\rm (i)]
\item 
the system of linear equations
\be
\label{eq:eigen2}
(M_{\pi} + (b_{ij} - b_{jj}) I_m) \bmat{x_1\\ \vdots \\ x_m} = \bmat{(b_{1j} - b_{ij}) + \de_{1j}(b_{ij} - b_{jj}) \\ \vdots \\ (b_{mj} - b_{ij}) + \de_{mj}(b_{ij} - b_{jj})}
\ee
is consistent; 
\item
for any solution $(d_1, \ldots, d_m)$ to \eqref{eq:eigen2}, the vector $[h_1, h_2]^\top$ defined by
$$
h_1 = - \frac{b_{ij}}{k+b_{ij}-b_{jj}} - \sum_{l = 1}^m \frac{|W_1 \cap V_l|}{|W_1|}(d_l - \de_{lj}),\, h_2 = - \frac{b_{ij}}{k+b_{ij}-b_{jj}} - \sum_{l = 1}^m \frac{|W_2 \cap V_l|}{|W_2|}(d_l - \de_{lj}),
$$
where $W_2=V\setminus W_1$, is a solution to the system of equations
$$
\bmat{a + b_{ij} - b_{jj} & k-a \\
b & k - b + b_{ij} - b_{jj}} \bmat{y_1\\ y_2} = \bmat{0 \\ 0};
$$
\item
in particular, if $b_{jj} - b_{ij}$ is neither $k$ nor  $a-b$, then 
$$
\sum_{l = 1}^m \frac{|W_1 \cap V_l|}{|W_1|}(d_l - \de_{lj}) = \sum_{l = 1}^m \frac{|W_2 \cap V_l|}{|W_2|}(d_l - \de_{lj}) = - \frac{b_{ij}}{k+b_{ij}-b_{jj}}.
$$
\end{enumerate}
\end{enumerate}
\end{corollary}
 
\begin{proof}
Since $W_1$ is an $(a, b)$-perfect set, $\tau = \{W_1,W_2\}$ is an equitable partition of $\Ga$ with 
$$
M_{\tau} = \bmat{c_{11} & c_{12} \\ c_{21} & c_{22}} = \bmat{a & k-a\\b & k-b}. 
$$

(a) Consider $(i, j) = (2, 1)$. In this case, we have $c_{ij} - c_{jj} = b-a$,
\[r_{ij} = \frac{c_{ij}}{k+c_{ij}-c_{jj}} = \frac{b}{k-a+b},\]
and \eqref{eq:2equ} becomes
\be
\label{eq:2equ-ps}
b x_1 + (k-a) x_2 = 0. 
\ee
One solution is $(d_1,d_2)=(0,0)$. Then 
\[h_{t} = \frac{|V_t \cap W_1|}{|V_t|} - r_{ij}
= \frac{|V_t \cap W_1|}{|V_t|} - \frac{b}{k-1+b}\]
for $1 \le t \le m$, as given in \eqref{eq:eigen}, and the result follows from Theorem \ref{thm:2equ} immediately.     

(b) The result follows from Theorem \ref{thm:2equ} by treating $\tau, \pi$ above as $\pi, \tau$ in Theorem \ref{thm:2equ} and noting that the eigenvalues of $M_{\tau}$ are $k$ and $a-b$. 
\qed
\end{proof}

Since 
\[\frac{|V_t \cap W_2|}{|V_t|} - \frac{k-a}{k-a+b} = -\left(\frac{|V_t \cap W_1|}{|V_t|} - \frac{b}{k-a+b}\right),\]
if we apply Theorem \ref{thm:2equ} to the case $n=2$ and $(i, j) = (1, 2)$, we will obtain the same result as in part (a) of Corollary \ref{coro:2equ-ps}.  

In the special case when $(a, b) = (0, 1)$ or $(1, 1)$, part (a) of Corollary \ref{coro:2equ-ps} gives \cite[Proposition 3]{E87} and \cite[Theorem 5.3]{Zhou2016}, respectively. This part also recovers the known result (see, for example, \cite{BCGG2019, G1993}) that any regular graph admitting an $(a, b)$-perfect set must have $a-b$ as an eigenvalue. In particular, when $(a, b) = (0, 1)$ or $(1, 1)$ it recovers the known results that any regular graph admitting a perfect code must have $-1$ as an eigenvalue and any regular graph admitting a total perfect code must have $0$ as an eigenvalue (see \cite[Corollary 5.4]{Zhou2016}). 

Obviously, $\pi = \{\{v\}: v \in V\}$ is an equitable partition of $\Ga$ whose quotient matrix $M_{\pi}$ is exactly the adjacency matrix $A$ of $\Ga$. Applying Theorem \ref{thm:2equ} to this trivial equitable partition, we obtain the following corollary. 

\begin{corollary}
\label{coro:2equ-tr}
Let $\Ga = (V, E)$ be a connected $k$-regular graph, where $k \ge 1$. Set $V = \{v_1, \ldots, v_m\}$, where $m = |V|$. Let $\tau = \{W_1, \ldots, W_n\}$ be an equitable partition of $\Ga$ with quotient matrix $M_{\tau} = (c_{ij})_{1 \le i, j \le n}$. Then, for any $1 \le i, j \le n$ with $i \ne j$, and any solution $(d_1, \ldots, d_n)$ to the system of linear equations \eqref{eq:2equ}, the vector $[h_1, \ldots, h_m]^\top$ defined by
\be
\label{eq:kt1}
h_t = -r_{ij} - \sum_{1 \le l \le n\atop v_t \in W_l} (d_l - \de_{lj})\;
\text{ for $1\le t\le m$}
\ee
is a solution to the system of equations
$$
(A + (c_{ij} - c_{jj}) I_m) \bmat{y_1\\ \vdots \\ y_m} = \bmat{0 \\ \vdots \\ 0}.
$$
In particular, either $(h_1, \ldots, h_m) = (0, \ldots, 0)$ or $c_{jj} - c_{ij}$ is an eigenvalue of $\Ga$ with $[h_1, \ldots, h_m]^\top$ as a corresponding eigenvector. 
\end{corollary}

In the case when $n = 2$, Corollary \ref{coro:2equ-tr} gives the following result. 

\begin{corollary}
\label{coro:2equ-tr-ps}
Let $\Ga = (V, E)$ be a connected $k$-regular graph, where $k \ge 1$. If $\Ga$ admits an $(a, b)$-perfect set $W_1$, where $0 \le a \le k-1$ and $1 \le b \le k$, then for any $\a \in \RRR \setminus \left\{1 - (k-a+b)^{-1}\right\}$, the vector $[h_v(\a)]^\top_{v \in V}$ defined by
$$
h_v(\a) = \left\{
\begin{array}{ll}
(k - a) \left(\displaystyle{\frac{1}{k-a+b}}-(1-\a)\right), & \text{if } v \in W_1 \\
- b \left(\displaystyle{\frac{1}{k-a+b}}-(1-\a)\right), & \text{if } v \not \in W_1
\end{array}
\right.
$$
is an eigenvector for the eigenvalue $a-b$ of $\Ga$.
\end{corollary}

\begin{proof}
Denote $V = \{v_1, \ldots, v_m\}$, where $m = |V|$. Let $\pi = \{\{v_t\}: t = 1, \ldots, m\}$ and $\tau = \{W_1, W_2\}$, with $W_2=V\setminus W_1$.

Consider $(i, j) = (2, 1)$. As seen in the proof of Corollary \ref{coro:2equ-ps}, in this case \eqref{eq:2equ} becomes \eqref{eq:2equ-ps}. The general solution is $(d_1, d_2) = ((1-\a)(k-a), -(1-\a)b)$, where $\a \in \RRR$. On the other hand, by \eqref{eq:kt1}, 
\[\begin{array}{rcll}
h_t(\a) &=& \displaystyle{\frac{k-a}{k-a+b}} - d_1 = (k-a)\left(\displaystyle{\frac{1}{k-a+b}}-(1-\a)\right)
& \hbox{if }v_t \in W_1,\\[3mm]
h_t(\a) &=& - \displaystyle{\frac{b}{k-a+b}} - d_2 = - b \left(\displaystyle{\frac{1}{k-a+b}}-(1-\a)\right)
& \hbox{if }v_t \in W_2.\\
\end{array}\]
Since $k-a \ge 1$ and $b \ge 1$, if 
\[\a \ne 1 - \frac{1}{k-a+b},\]
then $(h_1(\a), \ldots, h_m(\a)) \ne (0, \ldots, 0)$ and, by Corollary \ref{coro:2equ-tr}, $c_{11} - c_{21} = a-b$ is an eigenvalue of $\Ga$ and $[h_1(\a), \ldots, h_m(\a)]^\top$ is a corresponding eigenvector. 
\qed
\end{proof}

\section{Perfect sets in normal Cayley graphs}
\label{sec:eqCay}

Let $G$ be a group.
Recall that the Cayley graph $\Cay(G,S)$ is normal if the inverse-closed subset $S \subseteq G \setminus \{1\}$ is closed under conjugation; equivalently, the graph admits both left and right actions of $G$. It is connected if and only if
$S$ is a generating set for $G$. In this section we give necessary conditions for the existence of perfect sets in normal Cayley graphs using characters of groups and results from the previous section. Throughout this section, $\Cay(G,S)$
is a connected normal Cayley graph for $G$.

The graph $\Cay(G.S)$ is $k$-regular, where $k=|S|$. As in
Section~\ref{sec:eqReg}, $A$ is the adjacency matrix of this graph. 
The rows and columns of $A$ are indexed by the elements of $G$ such that the $(x,y)$-entry is $1$ or $0$ depending on whether $yx^{-1}$ is in $S$ or not. We assume throughout this section that $\tau = \{W_1, \ldots, W_n\}$ is an equitable partition of $\Cay(G, S)$ with quotient matrix $M_{\tau} = (c_{ij})_{1 \le i, j \le n}$. Set
$$
\tau x = \{W_1 x, \ldots, W_n x\}
$$
for any $x \in G$. For $X, Y \subseteq G$, set $X^{-1} = \{x^{-1}: x \in X\}$ and $XY = \{xy: x \in X, y \in Y\}$, as usual. 

\begin{lemma}
\label{lem:basic}
The following hold:
\begin{enumerate}[\rm (a)]
\item for every $x \in G$, $\tau x$ is an equitable partition of $\Cay(G, S)$, and moreover $M_{\tau x} = M_{\tau}$; 
\item for $1 \le i, j \le n$ and every $x \in W_i$, we have $|\{g \in S: x \in gW_j\}| = c_{ij}$;
\item for $1 \le i, j \le n$ and every $x \in W_i$, we have
$|\{g \in S: x \in W_{j}g\}| = c_{ij}$. 
\end{enumerate}
\end{lemma}

\begin{proof} (a) This follows from the fact that the right regular multiplication of $G$ is an automorphism of $\Cay(G, S)$. 

(b) This is a restatement of the definition of $c_{ij}$.

(c) Since $S$ is closed under conjugation, if $x = yg \in W_{j}g$ for some $y \in W_j$ and $g \in S$, then $g^{y^{-1}} \in S$ and $x = g^{y^{-1}}y \in g^{y^{-1}}W_j$. Moreover, if $x = y g = z h$ for some $y, z \in W_j$ and $g, h \in S$ with $g \ne h$, then $y \ne z$ and hence $g^{y^{-1}} \ne h^{z^{-1}}$ as $x = g^{y^{-1}}y = h^{z^{-1}}z$. From these and (b) we obtain (c) immediately. 
\qed
\end{proof}

By definition, an $(a, b)$-perfect set in $\Cay(G,S)$ is a subset $W_1$ of $G$ such that for any $x \in W_1$ there are exactly $a$ elements $g \in S$ with $gx \in W_1$ and for any $y \in W_2 = G \setminus W_1$ there are exactly $k-b$ elements $h \in S$ with $hy \in W_2$. In the case when $\tau = \{W_1, W_2\}$, Lemma \ref{lem:basic} gives rise to the following result. Note that parts (a) and (b) hold for arbitrary Cayley graphs. 

\begin{corollary}
\label{coro:basic}
If $W_1 \subseteq G$ is an $(a, b)$-perfect set in $\Cay(G, S)$, where $0 \le a \le k-1$ and $1 \le b \le k$, and $W_2=G\setminus W_1$, then the following hold:
\begin{enumerate}[\rm (a)]
\item for every $x \in G$, $W_1x$ is an $(a, b)$-perfect set in $\Ga$;
\item for every $x \in G$, $|\{g \in S: x \in gW_1\}|$ equals $a$ if $x \in W_1$ and equals $b$ if $x \in W_2$, and $|\{g \in S: x \in gW_2\}|$ equals $k-a$ if $x \in W_1$ and equals $k-b$ if $x \in W_2$;
\item for every $x\in G$, $|\{g \in S: x \in W_1g\}|$ equals $a$ if $x \in W_1$ and equals $b$ if $x \in W_2$, and $|\{g \in S: x \in W_2g\}|$ equals $k-a$ if $x \in W_1$ and equals $k-b$ if $x \in W_2$.
\end{enumerate}
\end{corollary}

By part (a) of Lemma \ref{lem:basic}, when considering an equitable partition $\tau$ of the vertex set in a Cayley graph we may assume without loss of generality that $1 \in W_1$. In particular, by part (a) of Corollary \ref{coro:basic}, when considering a perfect set $W_1$ in a Cayley graph we may assume without loss of generality that $1 \in W_1$. 

Recall that the left regular permutation representation of $G$,
$$
\l_G: G \rightarrow \Sym(G)
$$
given by $\l_G(x)=gx$, extends to a linear representation on $\CCC G$ defined by
$$
(\l_G(g)(f))(x) = f(g^{-1}x)\; \hbox{ for } x \in G, f\in
\CCC G.
$$
It can be verified that
$$
\l_G(g)(\d_v) = \d_{gv}\;\hbox{ for } g, v \in G.
$$
Let $\vp: \CCC G \rightarrow \CCC G$ be the adjacency linear mapping as in (\ref{eq:vphi}) for $V=G$ and $\Ga = \Cay(G, S)$. 

\begin{lemma}
\label{lem:vphi}
With $\varphi$ as above,
$\varphi = \sum_{g \in S} \l_G(g)$.
\end{lemma}

\begin{proof}
For any $v \in G$, we have
\[\left(\sum_{g \in S} \l_G(g)\right)(\d_v) = \sum_{g \in S} \l_G(g)(\d_v) = \sum_{g \in S} \d_{gv} = \sum_{u \in \Ga(v)} \d_u = \varphi(\d_v).\]
Since $\{\d_v: v \in G\}$ is a basis of $\CCC G$, the result follows. 
\qed
\end{proof}

Denote by $\otimes$ the Kronecker product of matrices and $\rank(M)$ the rank of a matrix $M$. In the following discussion, a vector is interpreted as an 
$n$-tuple of elements of $\CCC G$, that is, an element of $(\CCC G)^n$.

\begin{lemma}
\label{lem:nec-cond}
Let $A$ be the adjacency matrix of $\Ga = \Cay(G, S)$. Then for any equitable partition $\tau = \{W_1, \ldots, W_n\}$ of $\Ga$ we have 
\be
\label{eq:rk-ineq}
\rank\left(I_{n} \otimes A - M_{\tau}^{T} \otimes I_{|G|}\right) \le n|G| - r_{\tau},
\ee
where $r_{\tau}$ is the rank of the set of vectors $[\d_{W_{1}x}, \ldots , \d_{W_{n}x}]^\top$ for $x \in G$, in the linear space $(\CCC G)^n$. 
\end{lemma}

\begin{proof}
Write $M_{\tau} = (c_{ij})_{1 \le i, j \le n}$. By part (a) of Lemma \ref{lem:basic}, for any $x \in G$, $\tau x = \{W_1 x, \ldots, W_n x\} $ is an equitable partition of $\Cay(G, S)$ with quotient matrix $M_{\tau x} = M_{\tau}$. Hence, by Lemma \ref{lem:ps}, we have
$$
\vp(\d_{W_{j}x}) = \sum_{i=1}^n c_{ij}\d_{W_{i}x}\;\hbox{ for } 1 \le j \le n. 
$$
In other words, for any $x \in G$, $(\d_{W_{1}x}, \ldots , \d_{W_{n}x})$ is a solution to the linear equation system
\be
\label{eq:ns2}
\left(I_{n} \otimes A - M_{\tau}^{T} \otimes I_{|G|}\right) \bmat{\mathbf{x}_{1} \\ \vdots \\ \mathbf{x}_{n}} = \bmat{\mathbf{0} \\ \vdots \\ \mathbf{0}},
\ee
where $\mathbf{x}_{i} \in \CCC G$ for $1 \le i \le n$, and $\mathbf{0}$ is the zero vector of $\CCC G$. Note that here we have $n$ linear systems each with $|G|$ equations. 

By the fundamental theorem of linear algebra, the dimension of the solution space of \eqref{eq:ns2} is equal to $n|G| - \rank\left(I_{n} \otimes A - M_{\tau}^{T} \otimes I_{|G|}\right)$. Since, as seen above, the vectors $[\d_{W_{1}x}, \ldots , \d_{W_{n}x}]^\top$ for $x \in G$ form a subset of the solution space of \eqref{eq:ns2}, it follows from the definition of $r_{\tau}$ that $r_{\tau} \le n|G| - \rank\left(I_{n} \otimes A - M_{\tau}^{T} \otimes I_{|G|}\right)$, as claimed in \eqref{eq:rk-ineq}. 
\qed
\end{proof}

If we treat $I_{n} \otimes A - M_{\tau}^{T} \otimes I_{|G|}$ as a linear mapping $\Psi$ from $(\CCC G)^n$ to $(\CCC G)^n$, then the above proof of Lemma \ref{lem:nec-cond} essentially says that $[\d_{W_{1}x}, \ldots, \d_{W_{n}x}]^\top \in \Ker(\Psi)$ for every $x \in G$. Note that the $n$ linear systems \eqref{eq:ns2} are dependent as $|S| \d_{G} = A \d_{G} = \sum_{j=1}^{n}\sum_{i=1}^{n} c_{ij}\d_{W_{i}x}$. 

Now we bring in the assumption that $\Cay(G, S)$ is normal. Then its eigenvalues are given by
\[\l_{j} = \frac{1}{\chi_j(1)} \sum_{g \in S} \chi_j(g)\;\hbox{ for }1 \le j \le h
\]
(\cite{DS81, Z88}, see also \cite[Theorem 5]{LZ2022}), and moreover the multiplicity of $\l_{j}$ is equal to
\[\sum_{1 \le i \le h\atop\l_i = \l_j} \chi_i(1)^2,\]
where $\{\chi_1, \ldots, \chi_h\}$ is a complete set of irreducible characters of $G$. Recall that the degree of an irreducible character $\chi$ of $G$ is equal to $\chi(1)$. Since $\chi$ is a class function, as usual, for a conjugacy class $K$ of $G$ we write $\chi(K) =\chi(g)$ where $g \in K$. Denote the set of all conjugacy classes of $G$ by
\be
\label{eq:pi1}
\pi = \{K_1, \ldots, K_m\}.
\ee
We call this the \emph{conjugacy class partition} of $G$. 
Then $\CCC \pi$ is the vector space of class functions of $G$. Since $S$ is closed under conjugation, the inner automorphism group $\Inn(G)$ of $G$ is a subgroup of $\Aut(\Ga)$ with respect to its natural action on $G$. Since the orbits of $\Inn(G)$ on $G$ are precisely the conjugacy classes of $G$, it follows that $\pi$ is an equitable partition of $\Cay(G, S)$.  

The first statement in the next lemma is a well known result in group theory, while the second one was proved in \cite{E87} using \cite[(5.4)]{Feit}. 

\begin{lemma}
[{\cite[Lemma 5]{E87}}]
\label{lem:ev}
Let $\pi$ be as in \eqref{eq:pi1}. Then the irreducible characters of $G$ constitute a basis of $\CCC \pi$, and moreover they are eigenvectors of $\varphi_{\pi}$. More explicitly, for any irreducible character $\chi$ of $G$, 
$$
\varphi_{\pi}(\chi) = \left(\frac{1}{\chi(1)} \sum_{g \in S} \chi(g)\right) \chi.
$$
\end{lemma}

Now we are ready to prove the first main result in this section. 
 
\begin{theorem}
\label{thm:nec-cond}
Let $\Ga=\Cay(G, S)$ be a normal Cayley graph. Let $\pi$ be the conjugacy class partition of $G$, $A$ the adjacency matrix of $\Ga$, and $k = |S|$ the degree of $\Ga$. Let $a$ and $b$ be integers with $0 \le a \le k-1$ and $1 \le b \le k$, and let $\QQ$ be the set of inequivalent irreducible characters $\chi$ of $G$ such that $(1/\chi(1)) \sum_{g \in S} \chi(g) = a-b$. If $W_1 \subset G$ is an $(a, b)$-perfect set in $\Ga$, and $W_2=G\setminus W_1$, then the following statements hold:
\begin{itemize}
\item[\rm (a)] 
\be
\label{eq:rk}
\rank \bmat{A-a I & -bI \\ -(k-a)I & A-(k-b)I} \le 2|G| - r_{W_1},
\ee
where $I = I_{|G|}$ and $r_{W_1}$ is the rank of the set of vectors $[\d_{W_1x}, \d_{W_2x}]^\top$ (for $x \in G$) in $(\CCC G)^2$;  
\item[\rm (b)] 
letting $r$ be the rank of the set of vectors $\{\d_{W_1x}: x \in G\}$ in $\CCC G$, the multiplicity of the eigenvalue $a-b$ of $\Ga$ is at least $r-1$, which in turn is at least $|G|/|W_1|-1$;
\item[\rm (c)] 
\be
\label{eq:rk1}
\rank([|K_i| \chi(K_i)]_{1 \le i \le m, \chi \in \QQ}) \ge \rank([|K_i \cap W_1x|]_{1 \le i \le m, x \in G}) - 1.
\ee
In particular,
\be
\label{eq:rk2}
|\QQ| \ge \rank([|K_i \cap W_1x|]_{1 \le i \le m, x \in G}) - 1.
\ee
\end{itemize}
\end{theorem}

\begin{proof}
Since $W_1$ is an $(a, b)$-perfect set in $\Ga$, $\tau := \{W_1,W_2\}$ is an equitable partition of $\Ga$ with quotient matrix $M_{\tau} = \bmat{a & k-a\\b & k-b}$. Since $r_{W_1}$ takes the role of $r_{\tau}$ and 
$$
I_{2} \otimes A - M_{\tau}^{T} \otimes I = \bmat{A-a I & -bI \\ -(k-a)I & A-(k-b)I},
$$
where $I = I_{|G|}$, we obtain \eqref{eq:rk} from Lemma \ref{lem:nec-cond} immediately. 

By part (a) of Corollary \ref{coro:basic}, for any $x \in G$, $\tau x = \{W_1x, W_2x\}$ is an equitable partition in $\Ga$ with quotient matrix $M_{\tau x} = M_{\tau}$. Similarly to the proof of Lemma \ref{lem:nec-cond}, by Lemma \ref{lem:ps} we have $\vp(\d_{W_1x}) = a\d_{W_1x} + b\d_{W_2x}$ and $\vp(\d_{W_2x}) = (k-a)\d_{W_1x} + (k-b)\d_{W_2x}$ for $x \in G$. Using $\d_{W_1x}+\d_{W_2x} = \d_{G}$ and $A\d_{G} = k \d_{G}$, one can verify that both equations reduce to
$$
\vp(\d_{W_1x}) + (b-a)\d_{W_1x} = b \d_{G}.
$$  
Since $\{\d_{W_1x}: x \in G\}$ has rank $r$, it follows that the linear equation system $\vp(\mathbf{x}) + (b-a)\mathbf{x} = b \d_{G}$, where $\mathbf{x} \in \CCC G$, has at least $r$ independent solutions. Hence the solution space of the homogeneous system $\vp(\mathbf{x}) + (b-a)\mathbf{x} = \mathbf{0}
$ has dimension at least $r - 1$. In other words, 
\be
\label{eq:ker}
\dim(\Ker(\vp + (b-a)\Id)) \ge r - 1.
\ee

We now compute $\dim(\Ker(\vp + (b-a)\Id))$ using the irreducible characters of $G$. It is well known that in the decomposition of the regular representation into a direct sum of irreducible representations, the number of times an irreducible representation occurs is equal to its degree. Thus, by Lemmas \ref{lem:vphi} and \ref{lem:ev}, we have
$$
\vp = \bigoplus_{\rho}\left(\frac{1}{\chi_\rho(1)} \sum_{g \in S} \chi_\rho(g)\right) \mathrm{Id}_\rho,
$$
where the direct sum runs over all irreducible representations $\rho$ in the decomposition of $\l_G$ (into a direct sum of irreducible representations), and $\mathrm{Id}_\rho$ is the identity mapping on the space affording $\rho$. Hence
\be
\label{eq:dec1}
\vp + (b-a)\Id = \bigoplus_\rho\left(\frac{1}{\chi_\rho(1)} \sum_{g \in S} \chi_\rho(g) + (b-a)\right) \mathrm{Id}_\rho.
\ee
Therefore, by the definition of $\QQ$, we obtain $\dim(\Ker(\vp + (b-a)\Id)) = \sum_{\chi \in \QQ} \chi(1)^2$. Thus, by (\ref{eq:ker}), $\sum_{\chi \in \QQ} \chi(1)^2 \ge r - 1$. Since the left-hand side of this inequality is the multiplicity of $a-b$, we obtain that the multiplicity of the eigenvalue $a-b$ of $\Ga$ is at least $r-1$. Now by the definition of $r$, we may take $r$ linearly independent vectors $\d_{W_1x_1}, \ldots, \d_{W_1x_r}$ from $\{\d_{W_1x}: x \in G\}$. Then every vector in this set is a linear combination of these $r$ vectors. If $r < |G|/|W_1|$, then $|\bigcup_{i=1}^{r}W_1x_i| \le r|W_1| < |G|$, so we can take an element $y \in G \setminus (\bigcup_{i=1}^{r}W_1x_i)$. Take an element $z \in W_1$ and set $x_{r+1} = z^{-1}y$, so that $y \in W_1x_{r+1}$. Since $\d_{W_1x_{r+1}}$ is $1$ at $y$ but all $\d_{W_1x_1}, \ldots, \d_{W_1x_r}$ are $0$ at $y$, $\d_{W_1}x_{r+1}$ cannot be a linear combination of $\d_{W_1x_1}, \ldots, \d_{W_1x_r}$, which is a contradiction. Therefore, $r \ge |G|/|W_1|$ and part (b) is proved. 

It remains to prove \eqref{eq:rk1}. Applying part (a) of Corollary \ref{coro:2equ-ps} to $\pi$ and $\tau x$, we obtain 
$$
f = \sum_{i=1}^m h_i \d_{K_i} \in \Ker(\vp_{\pi} + (b-a)\Id_{\pi}),
$$
where
$$
h_i = \frac{|K_i \cap W_1x|}{|K_i|} - \frac{b}{k-a+b}\;\hbox{ for } 1 \le i \le m.
$$
Note that, since the rational numbers $h_i$ depend on $x$, so does $f$. On the other hand, by \eqref{eq:dec1}, $\Ker(\vp_{\pi} + (b-a)\Id_{\pi})$ is spanned by the irreducible characters in $\QQ$. So for each pair $(\chi, x)$ there exists $a_{\chi, x} \in \CCC$ such that
\bean
f & = & \sum_{\chi \in \QQ} a_{\chi, x}\chi \\
& = & \sum_{\chi \in \QQ} a_{\chi, x} \left(\sum_{i=1}^m \chi(K_i) \d_{K_i}\right) \\
& = & \sum_{i=1}^m \left(\sum_{\chi \in \QQ} a_{\chi, x} \chi(K_i)\right) \d_{K_i}.  
\eean
Since $\{\d_{K_i}: 1 \le i \le m\}$ is a basis of $\CCC \pi$, we have $h_i = \sum_{\chi \in \QQ} a_{\chi, x} \chi(K_i)$ for $1 \le i \le m$; that is,
$$
|K_i \cap W_1x| = \frac{b}{k-a+b} |K_i| + \sum_{\chi \in \QQ} a_{\chi, x}  |K_i| \chi(K_i),\;\, 1 \le i \le m.
$$
Since this holds for all $x \in G$, we obtain
$$
[|K_i \cap W_1x|]_{1 \le i \le m, x \in G} = \frac{b}{k-a+b} \bmat{|K_1| & \cdots & |K_1|\\ \vdots &  & \vdots \\ |K_m| & \cdots & |K_m|} + [|K_i| \chi(K_i)]_{1 \le i \le m, \chi \in \QQ} \cdot [a_{\chi, x}]_{\chi \in \QQ, x}.
$$
Since the first term on the right-hand side has rank $1$, the second term on the right-hand side should have rank at least $\rank([|K_i \cap W_1x|]_{1 \le i \le m, x \in G}) - 1$. On the other hand, the rank of $[|K_i| \chi(K_i)]_{1 \le i \le m, \chi \in \QQ}$ is no less than that of the second term on the right-hand side above. Hence \eqref{eq:rk1} follows. We obtain \eqref{eq:rk2} from \eqref{eq:rk1} as $|\QQ| \ge \rank([|K_i| \chi(K_i)]_{1 \le i \le m, \chi \in \QQ})$.
\qed
\end{proof}

In essence, \eqref{eq:rk2} says that the number of inequivalent irreducible characters of $G$ that give rise to the eigenvalue $a-b$ of $\Ga$ is no less than $\rank([|K_i \cap Cx|]_{1 \le i \le m, x \in G}) - 1$.

In the special case when $G$ is abelian, it has exactly $|G|$ irreducible characters $\chi$ all of which are linear. Thus $\QQ$ consists of those $\chi$ such that $\sum_{g \in S} \chi(g) = a-b$. Since $G$ is abelian, the eigenvalues of $\Cay(G, S)$ are precisely $\sum_{g \in S} \chi(g)$, with $\chi$ running over all irreducible characters of $G$. Therefore, part (b) of Theorem \ref{thm:nec-cond} gives the following result for Cayley graphs of abelian groups. (Compare this result with Corollary \ref{coro:2equ-tr-ps}.)

\begin{corollary}
\label{cor:nec-cond}
Let $G$ be an abelian group and $S$ an inverse-closed subset of $G$ with $1 \not \in S$. If $\Cay(G, S)$ admits an $(a, b)$-perfect set $W_1$, where $0 \le a \le k-1$ and $1 \le b \le k$ with $k = |S|$ the degree of $\Cay(G, S)$, then the multiplicity of $a-b$ as an eigenvalue of $\Cay(G, S)$ is at least $r-1$, where $r$ is the rank of the set of vectors $\{\d_{W_1x}: x \in G\}$ in $\CCC G$.   
\end{corollary}

Now let us consider the special case when $(a, b) = (1, 1)$ (that is, $W_1$ is a total perfect code in $\Ga$ with $1 \in W_1$). Let $\pi$ be as in \eqref{eq:pi1}. Without loss of generality we may assume $S = \bigcup_{i=1}^s K_i$ and $K_{s+1} = \{1\}$, where $1 \le s \le m-1$. Take $g_j \in K_j$ for $1 \le j \le s+1$ (so that $g_{s+1} = 1$). Since the neighbourhood of $1$ in $\Cay(G, S)$ is $S$, one can see that there exists a unique $i^*$ with $1 \le i^* \le s$ such that $|K_{i^*} \cap W_1| = 1$ and $|K_{i} \cap W_1| = 0$ for $1 \le i \le s$ with $i \ne i^*$. It can be proved (see the proof of Theorem 5.5 in \cite{Zhou2016}) that the first $s+1$ rows of the matrix $[|K_i \cap W_1g_j|]_{1 \le i \le m, 1 \le j \le s+1}$ form the submatrix
$\bmat{I_{s} & e_{i^*} \\ 0 & 1}$ whose rank is equal to $s+1$, where $e_{i^*} = [0, \ldots, 0, 1, 0,  \ldots, 0]^\top$ with $1$ in the $i^*$-th coordinate. So $[|K_i \cap W_1g_j|]_{1 \le i \le m, 1 \le j \le s+1}$ has rank at least $s+1$. Since this matrix has $s+1$ columns, it follows that its rank is equal to $s+1$. Thus, $\rank([|K_i \cap W_1x|]_{1 \le i \le m, x \in G}) \ge \rank([|K_i \cap W_1g_j|]_{1 \le i \le m, 1 \le j \le s+1}) = s+1$. By \eqref{eq:rk1} and \eqref{eq:rk2}, we then obtain $\rank([|K_i| \chi(K_i)]_{1 \le i \le m, \chi \in \QQ}) \ge s$ and $|\QQ| \ge s$, the latter being part (b) of \cite[Theorem 5.5]{Zhou2016}. Since $(a, b) = (1, 1)$, by part (c) of Corollary \ref{coro:basic}, $\{W_1g: g \in S\}$ is a partition of $G$ (see also \cite[Lemma 3.1(c)]{Zhou2016}). It follows that the rank of $\{\d_{W_1g}: g \in S\}$ is equal to $|S|$ and consequently the rank of $\{\d_{W_1x}: x \in G\}$ is no less than $|S|$. Thus, by part (b) of Theorem \ref{thm:nec-cond}, $\sum_{\chi \in \QQ} \chi(1)^2 \ge |S| - 1$, which is exactly part (a) of \cite[Theorem 5.5]{Zhou2016}. In summary, in the special case when $(a, b) = (1, 1)$, parts (b) and (c) of Theorem \ref{thm:nec-cond} yield the following known result. 

\begin{corollary}
[{\cite[Theorem 5.5]{Zhou2016}}]
\label{coro:nec-a}
If a normal Cayley graph $\Cay(G, S)$ admits a total perfect code, then the multiplicity of its eigenvalue $0$ is at least $|S| - 1$, and the number of inequivalent irreducible characters of $G$ which give rise to this eigenvalue is no less than the number of conjugacy classes of $G$ contained in $S$.
\end{corollary}

Similarly, in the case when $(a, b) = (0, 1)$, parts (b) and (c) of Theorem \ref{thm:nec-cond} give the following known result due to G. Landauer (1979) according to \cite{E87}. 

\begin{corollary}
[{\cite[Theorem 6]{E87}}]
\label{coro:nec-b}
If a normal Cayley graph $\Cay(G, S)$ admits a perfect code, then the multiplicity of its eigenvalue $-1$ is at least $|S|$, and the number of inequivalent irreducible characters of $G$ that give rise to this eigenvalue is no less than the number of conjugacy classes of $G$ contained in $S$. 
\end{corollary}

Let $\Ga = \Cay(G, S)$ be a normal Cayley graph. Let $H$ be a subgroup of $G$ and let 
\be
\label{eq:piH}
\pi_H = \{x_1 H, \ldots, x_m H\}
\ee
be the set of left cosets of $H$ in $G$, where $m = |G:H|$ and $x_1, \ldots, x_m$ are representatives of such cosets. As noted in \cite{E87,Zhou2016}, $\pi_H$ is an equitable partition of $\Ga$ because $H$ can be viewed as a subgroup of $\Aut(\Ga)$ and the left cosets of $H$ in $G$ are the $H$-orbits under the action of $H$ on $G$ by right multiplication. It is evident that $\{\d_{x_1 H}, \ldots, \d_{x_m H}\}$ is a basis of $\CCC \pi_H$. Let 
$$
\l_G^H: G \rightarrow \Sym(\CCC \pi_H)
$$ 
be the representation of $G$ defined by $\l_G^H(g)(\d_{x_i H}) = \d_{gx_i H}$ for $g \in G$ and $1 \le i \le m$. 

The second main result in this section is as follows. 

\begin{theorem}
\label{thm:nec-cd}
Let $\Ga = \Cay(G, S)$ be a normal Cayley graph. Let $a$ and $b$ be integers with $0 \le a \le k-1$ and $1 \le b \le k$, where $k = |S|$ is the degree of $\Ga$. Let $H$ be a subgroup of $G$ such that 
\be
\label{eq:noteq}
\frac{1}{\chi(1)} \sum_{g \in S} \chi(g) \ne a-b 
\ee
for every irreducible character $\chi$ of $G$ which appears in the decomposition of $\l_G^H$ into a direct sum of irreducible characters. Then any $(a, b)$-perfect set $W_1$ in $\Ga$ contains the same number of vertices from each left coset of $H$ in $G$; more specifically, for any $x \in G$, we have 
\be
\label{eq:nec-cd}
\frac{|x H \cap W_1|}{|H|} = \frac{b}{k-a+b}
\ee
and hence $k-a+b$ must be a divisor of $b|H|$. In particular, if there exists a subgroup $H$ of $G$ satisfying \eqref{eq:noteq} such that $k-a+b$ does not divide $b|H|$, then $\Ga$ admits no $(a, b)$-perfect sets. 
\end{theorem}

\begin{proof}
As before, let $\pi = \pi_H$ be as in \eqref{eq:piH}. Let $\vp: \CCC G \rightarrow \CCC G$ be the linear mapping defined in (\ref{eq:vphi}) for $V=G$ and $\Ga = \Cay(G, S)$. As noticed in \cite{E87}, similarly to Lemma \ref{lem:vphi}, a straightforward computation yields
$$
\vp_{\pi} = \sum_{g \in S} \l_{G}^{H}(g).
$$
In other words, if $\l_{G}^{H}(g)$ is viewed as the permutation matrix of the permutation $\d_{x_i H} \mapsto \d_{gx_i H}$, $1 \le i \le m$ of the basis $\{\d_{x_1 H}, \ldots, \d_{x_m H}\}$ of $\CCC \pi$, then the matrix of $\vp_{\pi}$ is given by $M_{\pi} = \sum_{g \in S} \l_{G}^{H}(g)$. Using this, one can show that
$$
\vp_{\pi} = \bigoplus_{\rho}\left(\frac{1}{\chi_\rho(1)} \sum_{g \in S} \chi_\rho(g)\right) \mathrm{Id}_\rho,
$$
where the direct sum runs over all irreducible representations $\rho$ in the decomposition of $\l_G^H$ into a direct sum of irreducible representations. Hence
$$
\vp_{\pi} + (b-a)\Id_{\pi} = \bigoplus_{\rho}\left(\frac{1}{\chi_\rho(1)} \sum_{g \in S} \chi_\rho(g) + (b-a)\right) \mathrm{Id}_\rho,
$$
where $\rho$ runs over the same range as above. Since by our assumption $H$ satisfies \eqref{eq:noteq}, it follows that $\det(M_{\pi} + (b-a)I_m) \ne 0$ and hence the system of homogeneous equations \eqref{eq:eigen1} has trivial solution only. Thus, by part (a) of Corollary \ref{coro:2equ-ps}, for any $(a, b)$-perfect set $W_1$ in $\Ga$, we have
\[\frac{|x_i H \cap W_1|}{|x_i H|} = \frac{b}{k-a+b}\]
for $1 \le i \le m$. In other words, \eqref{eq:nec-cd} holds for any $x \in G$. 
\qed
\end{proof}

In the special case when $(a, b) = (0, 1)$ or $(1, 1)$, Theorem \ref{thm:nec-cd} gives the following two known results, respectively. 

\begin{corollary}
[{\cite[Theorem 7]{E87}}]
\label{coro:nec-c}
Let $\Ga = \Cay(G, S)$ be a normal Cayley graph. Let $H$ be a subgroup of $G$ such that $\sum_{g \in S} \chi(g) \ne -\chi(1)$ for every irreducible character $\chi$ of $G$ which appears in the decomposition of $\l_G^H$ into a direct sum of irreducible characters. Then any perfect code in $\Ga$ intersects every left coset of $H$ in $G$ at exactly $|H|/(|S|+1)$ elements. In particular, if $|S|+1$ is not a divisor of $|H|$, then $\Ga$ has no perfect codes. 
\end{corollary}

\begin{corollary}
[{\cite[Theorem 5.10]{Zhou2016}}]
\label{coro:nec-d}
Let $\Ga = \Cay(G, S)$ be a normal Cayley graph. Let $H$ be a subgroup of $G$ such that $\sum_{g \in S} \chi(g) \ne 0$ for every irreducible character $\chi$ of $G$ which appears in the decomposition of $\l_G^H$ into a direct sum of irreducible characters. Then any total perfect code in $\Ga$ intersects every left coset of $H$ in $G$ at exactly $|H|/|S|$ elements. In particular, if $|S|$ is not a divisor of $|H|$, then $\Ga$ has no total perfect codes. 
\end{corollary}

If $H = \{1\}$ is the trivial subgroup of $G$, then $\pi_H$ defined in \eqref{eq:piH} is the trivial partition of $G$ into singletons and $\l_G^H$ can be identified with $\l_G$. For this special choice of $H$, Theorem \ref{thm:nec-cd} gives the known result that $\Cay(G, S)$ admits no $(a, b)$-perfect sets unless it has $a-b$ as an eigenvalue.

\section{Further discussions}
\label{sec:rem}

A graph is vertex-transitive if its automorphism group is transitive on its vertex set. The lower bounds in part (b) of Theorem \ref{thm:nec-cond} can be generalized to vertex-transitive graphs. We present this result in the following theorem whose proof in the special case of normal Cayley graphs yields another proof of part (b) of Theorem \ref{thm:nec-cond}. In this part, we denote by $u^x$ the image of the vertex $u$ under the automorphism $x$.

\begin{theorem}
\label{thm:vt}
Let $\Ga = (V, E)$ be a vertex-transitive graph with degree $k$, and let $a$ and $b$ be integers with $0 \le a \le k-1$ and $1 \le b \le k$. If $\Ga$ admits an $(a, b)$-perfect set $W_1$, then the multiplicity of $a-b$ as an eigenvalue of $\Ga$ is at least $r-1$, which in turn is at least $|V|/|W_1| - 1$, where $r$ is the rank of the set of vectors $\{\d_{W_1^x}: x \in \Aut(\Ga)\}$ in $\CCC V$, with $W_1^x$ the image of $W_1$ under $x$.
\end{theorem}

\begin{proof}
Since $\Ga$ is $k$-regular, $k$ is an eigenvalue of $\Ga$. Since $\Ga$ is vertex-transitive and $W_1$ is an $(a, b)$-perfect set in $\Ga$, for any $x \in \Aut(\Ga)$, $W_1^x := \{u^x: u \in W_1\}$ is an $(a, b)$-perfect set in $\Ga$. Thus, by Lemma \ref{lem:ps}, $\vp(\d_{W_1^x}) = (a-b)\d_{W_1^x} + b \d_{V}$, where $\vp$ is as defined in \eqref{eq:vphi}. Similarly to the argument leading to \eqref{eq:ker}, we see that the homogeneous linear system $\vp(\mathbf{x}) + (b-a)\mathbf{x} = \mathbf{0}$ has at least $r - 1$ linearly independent solutions. In other words, the multiplicity of $a-b$ as an eigenvalue of $\Ga$ is at least $r-1$. Take $r$ linearly independent vectors $[\d_{W_1^{x_1}}, \ldots, \d_{W_1^{x_r}}]^\top$, where $x_1, \ldots, x_r \in \Aut(\Ga)$. If $r < |V|/|W_1|$, then $|\bigcup_{i=1}^r W_1^{x_i}| \le r |W_1| < |V|$, so we can take $v \in V \setminus (\bigcup_{i=1}^r W_1^{x_i})$. Take $u \in W_1$. Since $\Ga$ is vertex-transitive, there exists an automorphism $x_{r+1} \in \Aut(\Ga)$ such that $u^{x_{r+1}} = v$. Since $u \in W_1$, we have $v \in W_1^{x_{r+1}}$ and hence $\d_{W_1^{x_{r+1}}}$ is $1$ at vertex $v$. On the other hand, by the choice of $v$, all $\d_{W_1^{x_1}}, \ldots, \d_{W_1^{x_r}}$ are $0$ at vertex $v$. Hence the vectors $\d_{W_1^{x_1}}, \ldots, \d_{W_1^{x_r}}, \d_{W_1^{x_{r+1}}}$ are linearly independent, but this contradicts the definition of $r$. Therefore, $r \ge |V|/|W_1|$. 
\qed
\end{proof}

We conclude this paper by proving the following proposition with the help of Theorem \ref{thm:nec-cond}. (We assume $n \ge 13$ to avoid lengthy discussions for small cases.) 

\begin{proposition}
\label{ex:dih}
Let $D_{2n} = \langle r, s: r^n = s^2 = 1, s^{-1}rs = r^{-1}\rangle$ be the dihedral group of order $2n$, and let $\Ga = \Cay(D_{2n}, S)$, where $S = \{r, r^{-1}, r^j s: 0 \le j \le n-1\}$. If $n \geq 13$ is odd, then for any integers $a, b$ with $0 \le a \le n+1$, $1 \le b \le n+2$ and $(a, b) \ne (2, n)$, $\Ga$ admits no $(a,b)$-perfect sets, and moreover the only $(2,n)$-perfect sets of $\Ga$ are $\{1, r^i, r^{-i}: 1 \le i \le (n-1)/2\}$ and $\{r^j s: 0 \le j \le n-1\}$. 
\end{proposition}

\begin{proof}
Suppose that $n \geq 13$ is odd. Then $D_{2n}$ has $(n+3)/2$ conjugacy classes (see, for example, \cite{Feit}), namely $\{1\}$, $\{r^i, r^{-i}\}$ ($1 \le i \le (n-1)/2$), and $\{r^j s: 0 \le j \le n-1\}$. Hence $\Ga$ is a connected normal Cayley graph with degree $|S| = n+2$. Since $n$ is odd, the irreducible characters of $D_{2n}$ are: the trivial character $\chi_1$; the linear character $\chi_2$ which takes values $\chi_2(1) = 1$, $\chi_2(r^i) = 1$ ($1 \le i \le (n-1)/2$) and $\chi_{2}(s) = -1$; and $(n-1)/2$ characters $\psi_{j}$ ($1 \le j \le (n-1)/2$) of degree $2$ which take values $\psi_{j}(1) = 2$, $\psi_{j}(r^i) = \omega^{ji} + \omega^{-ji}$ ($1 \le i \le (n-1)/2$) and $\psi_{j}(s) = 0$, where $\omega = e^{2\pi i/n}$. The eigenvalues of $\Ga$ are:
\begin{center}
\begin{tabular}{ll}
$n+2$ & with multiplicity 1,\\
$(1/\chi_2(1)) \sum_{g \in S} \chi_2(g) = 2-n$ & with multiplicity 1,\\
$(1/\psi_{j}(1)) \sum_{g \in S} \psi_{j}(g)$ & ($1 \le j \le (n-1)/2$) \\
$\quad{} = \omega^{j} + \omega^{-j} = 2 \cos (2\pi j/n)$ & each with multiplicity $4$.\\
\end{tabular}
\end{center}
Since $n$ is odd and $1 \le j \le (n-1)/2$, $2 \cos (2\pi j/n)$ is an integer if and only if $n = 3j$, in which case $3$ divides $n$ and $2\cos(2\pi j/n) = -1$. Note that $-(n+2) \le a-b \le n$, so $a-b \ne n+2$. Thus, if $a-b$ is an eigenvalue of $\Ga$, then it must be $2-n$ or $-1$, the latter happening only when $3$ divides $n$. Clearly, $a-b = 2-n$ if and only if $(a, b) \in \{(0, n-2), (1, n-1), (2, n), (3, n+1), (4, n+2)\}$. So, if $3$ is not a divisor of $n$ and $(a, b)$ is not one of these pairs, then $\Ga$ admits no $(a, b)$-perfect sets. If $3$ is a divisor of $n$, $(a, b)$ is not one of these pairs, and $a-b \neq -1$, then $\Ga$ admits no $(a, b)$-perfect sets.

Now assume that $(a, b) \in \{(0, n-2), (1, n-1), (2, n), (3, n+1), (4, n+2)\}$ and $\Ga$ admits an $(a,b)$-perfect set $W_1$. Then $|W_1| = b$ by \eqref{eq:abd} and $|W_2| = 2n-b = n+2-a$ as $a-b = 2-n$, implying that every vertex in $W_1$ is adjacent to every vertex in $W_2$. If $(a, b) \in \{(0, n-2), (1, n-1)\}$, then by part (b) of Theorem \ref{thm:nec-cond}, the multiplicity of $2-n$ as an eigenvalue of $\Ga$ is at least $\lceil 2n/b \rceil - 1 = 2$, which contradicts the fact that $2-n$ is a simple eigenvalue of $\Ga$. Hence $(a, b) \in \{(2, n), (3, n+1), (4, n+2)\}$. 

Consider the case $(a, b) = (2,n)$ first. In this case we have $|W_1| = |W_2| = n$ and the quotient matrix of the partition $\{W_1,W_2\}$ is $\bmat{2 & n \\ n & 2}$. Without loss of generality we may assume $1 \in W_1$. Then $W_2 \subseteq S$ as all possible edges between $W_1$ and $W_2$ exist in $\Ga$. Since $|W_2| = n$ and $|S|=n+2$, $|W_2 \cap \{r^j s: 0 \le j \le n-1\}| = n-2, n-1$ or $n$, that is, $|W_2 \cap \{r, r^{-1}\}| = 2, 1$ or $0$. Note that $r$ is adjacent to $r^l s r = r^{l-1} s $ for $0 \le l \le n-1$. So, if $|W_2 \cap \{r, r^{-1}\}| = 2$ or $1$, say, $r \in W_2$, then $r$ has at least $n-3 \ge 10$ neighbours in $W_2$, which is a contradiction. So we must have $W_2 \cap \{r, r^{-1}\} = \emptyset$, that is, $W_2 = \{r^j s: 0 \le j \le n-1\}$ and $W_1 = \{1, r^i, r^{-i}: 1 \le i \le (n-1)/2\}$. It can be verified that these two sets are both $(2,n)$-perfect sets in $\Ga$.

Now consider the case $(a, b) = (3,n+1)$. In this case we have $|W_1| = n+1$, $|W_2| = n-1$, and the quotient matrix of the partition $\{W_1, W_2\}$ is $\bmat{3 & n-1 \\ n+1 & 1}$. Assume that $1 \in W_1$.  Then $|W_2 \cap S| = n-1$ and hence $W_2 \subset S$.  If $W_2 \cap \{r, r^{-1}\} \neq \emptyset$, say, $r \in W_2$, then $r$ has at least $n-3 \ge 10$ neighbours in $W_2$, which is a contradiction. So $\{r, r^{-1}\} \subset W_1$. Thus, $W_1 = \{1, r^{j_1}s\} \cup \{r^i, r^{-i}: 1 \le i \le (n-1)/2$ and $W_2 = \{r^j s: 0 \le j \le n-1, j \ne j_1\}$ for some $0 \le j_1 \le n-1$. Since $n \ge 13$, there exists $0 \le j \le n-1$ with $j \ne j_0$ such that $j \pm 1 \not \equiv j_0 \mod n$. Then $r^{j}s \in W_2$ has two neighbours in $W_2$, namely $r r^{j}s = r^{j+1}s$ and $r^{-1} r^{j}s = r^{j-1}s$, but this is a contradiction. Assume that $1 \in W_2$. Then $|W_2 \cap S| = 1$ and so $W_1 = S \setminus W_2$. If $W_2 \cap S = \{r\}$ or $\{r^{-1}\}$, then $r$ or $r^{-1}$ has at least two neighbours in $W_2$, a contradiction. If $W_2 \cap S = \{r^{j_1}s\}$ for some $0 \le j_1 \le n-1$, then $W_1 = S \setminus \{r^{j_1}s\}$ and $W_2 = \{1, r^{j_1}s\} \cup \{r^i, r^{-i}: 2 \le i \le (n-1)/2\}$. Similarly to the argument in the previous paragraph, we see that $r \in W_1$ has at least $n-2 \ge 11$ neighbours in $W_1$, a contradiction. Hence $(a, b) \neq (3,n+1)$. 

Finally, consider the case $(a, b) = (4,n+2)$. In this case we have $|W_1| = n+2$, $|W_2| = n-2$, and the quotient matrix of the partition $\{W_1,W_2\}$ is $\bmat{4 & n-2 \\ n+2 & 0}$. So $W_2$ is an independent set of $\Ga$. If $1 \in W_2$, then $W_1 = S$ and $r \in W_1$ has $n \ge 13$ neighbours $r^{j}s$ ($0 \le j \le n-1$) in $W_1$, a contradiction. If $1 \in W_1$, then $|W_2 \cap S| = n-2$ and so $W_2 \subset S$. If $r$ or $r^{-1}$ is in $W_2$, then it has at least $n-6 \ge 7$ neighbours in $W_2$, a contradiction. So $\{r, r^{-1}\} \subset W_1$ and one can show that each vertex in $W_2$ has at least $n-5 \ge 8$ neighbours in $W_2$, again a contradiction. Hence $(a,b) \neq (4,n+2)$.

In the remainder of this proof we assume that $3$ is a divisor of $n$ and $a-b = -1$. Our goal is to prove that $\Ga$ admits no $(b-1,b)$-perfect sets. Suppose for a contradiction that $W_1$ is a $(b-1,b)$-perfect set in $\Ga$. Then $(n+3)|W_1| = 2nb$ by \eqref{eq:abd}, so $n+3$ divides $2nb$ and $|W_1| = 2nb/(n+3)$. If $1 \le b < (n+3)/5$, then by part (b) of Theorem \ref{thm:nec-cond}, the multiplicity of $-1$ as an eigenvalue of $\Ga$ is at least $\lceil 2n/|W_1| \rceil - 1 = \lceil (n+3)/b \rceil - 1 > 4$, which is a contradiction. Hence $b \ge (n+3)/5$. Since $W_2$ is an $(n+2-b, n+3-b)$-perfect set in $\Ga$, the same argument applied to $W_2$ yields $n+3-b \ge (n+3)/5$. Therefore, $(n+3)/5 \le b \le 4(n+3)/5$. Since $n \ge 13$, we have $b > 3$, so $b \ge 6$ as $b$ is divisible by $3$. 

Set $T = \{r^i, r^{-i}: 1 \le i \le (n-1)/2\}$ and $R = \{r^j s: 0 \le j \le n-1\}$. Then $\Ga(1) = S = \{r, r^{-1}\} \cup R$. The subgraph $\Ga[\{1\} \cup T]$ of $\Ga$ induced by $\{1\} \cup T$ is the cycle $(1, r, r^2, \ldots, r^{n-1}, 1)$, the subgraph $\Ga[R]$ of $\Ga$ induced by $R$ is the cycle $(s, rs, r^2 s, \ldots, r^{n-1}s, s)$, and every vertex in $\{1\} \cup T$ is adjacent to every vertex in $R$.  

\smallskip
\textsf{Case 1.} $1 \in W_1$. In this case, we have $|W_1 \cap S| = b-1$ and $|W_2 \cap S| = n+3-b$. If $r \in W_1$ and $r^{-1} \in W_2$, then for any $r^i \in W_1 \cap T$, we have
\[1 + |W_1 \cap R| = |\Ga(1) \cap W_1| = |\Ga(r^i) \cap W_1| = |\Ga(r^i) \cap ((W_1 \cap T) \cup \{1\})| + |W_1 \cap R|\]
and so $|\Ga(r^i) \cap ((W_1 \cap T) \cup \{1\})| = 1$, and for any $r^i \in W_2 \cap T$,
\[n+2-b = |\Ga(r^i) \cap W_2| = |\{r^{i-1}, r^{i+1}\} \cap (W_2 \cap T)| + |W_2 \cap R|.\]
Since $\{1\} \cup T$ induces the cycle $(1, r, r^2, \ldots, r^{n-1}, 1)$, it follows that for each $i \equiv 0 \mod 4$ we have $r^i, r^{i+1} \in W_1$ and $r^{i+2}, r^{i+3} \in W_2$. Since $r^{n-1} \in W_2$, we must have $n \equiv 0 \mod 4$, which is a contradiction as $n$ is odd. If $r \in W_2$ and $r^{-1} \in W_1$, we can get the same contradiction by a similar argument. 

If $r \in W_1$ and $r^{-1} \in W_1$, then for any $r^i \in W_1 \cap T$, we have
\[2 + |W_1 \cap R| = |\Ga(1) \cap W_1| = |\Ga(r^i) \cap W_1| = |\Ga(r^i) \cap ((W_1 \cap T) \cup \{1\})| + |W_1 \cap R|\]
and so $|\Ga(r^i) \cap ((W_1 \cap T) \cup \{1\})| = 2$. Since $\{1\} \cup T$ induces the cycle $(1, r, r^2, \ldots, r^{n-1}, 1)$, it follows that $T \subseteq W_1$ and $|W_1 \cap R| = b-3 \ge 3$. So $W_2 \cap T = \emptyset$ and $W_2 \cap R = W_2 \neq \emptyset$. Since $\Ga[R]$ is an $n$-cycle and both $W_1 \cap R$ and $W_2 \cap R$ are nonempty, there exists an edge of $\Ga[R]$ joining a vertex $r^{i}s \in W_1 \cap R$ and a vertex $r^{j}s \in W_2 \cap R$. Therefore, $b-1 = |\Ga(r^i s) \cap W_1| = n +  |\Ga(r^i s) \cap W_1 \cap R| = n+1$, and $n+2-b = |\Ga(r^j s) \cap W_2 \cap R| = 1$, but these two equations contradict each other. 

If $r \in W_2$ and $r^{-1} \in W_2$, then 
$|W_1 \cap R| = |\Ga(1) \cap W_1| = b-1$.
Assume that $W_1 \cap T \neq \emptyset$. Then, for each $r^i \in W_1 \cap T$, we have $|\Ga(r^i) \cap W_1| = |\Ga(1) \cap W_1| = |W_1 \cap R|$ and hence $\Ga(r^i) \cap W_1 \cap T = \emptyset$. Thus, along the path segment $(r, r^2, \ldots, r^{n-1})$ of the cycle $(1, r, r^2, \ldots, r^{n-1}, 1)$ the vertices must be in $W_2 \cap T$ and $W_1 \cap T$ alternately, which implies that $n$ is even, a contradiction. Hence $W_1 \cap T = \emptyset$. So, each vertex in $W_1 \setminus \{1\} = W_1 \cap R$ has at most two neighbours in $W_1$, but $1$ has $b-1 \ge 5$ neighbours in $W_1$, again a contradiction. 

\smallskip

\textsf{Case 2.} $1 \in W_2$. In this case, we have $|W_2 \cap S| = n+2-b$ and $|W_1 \cap S| = b$. If $r \in W_1$ and $r^{-1} \in W_2$, then for any $r^i \in W_1\cap T$,
\[1 + |W_1 \cap R| = |\Ga(1) \cap W_1| = |\Ga(r^i) \cap W_1| = |\Ga(r^i) \cap ((W_1 \cap T) \cup \{1\})| + |W_1 \cap R|,\]
so $|\Ga(r^i) \cap ((W_1 \cap T) \cup \{1\})| = 1$, and for any $r^i \in W_2 \cap T$,
\[n+2-b = |\Ga(r^i) \cap W_2| = |\{r^{i-1}, r^{i+1}\} \cap W_2 \cap T)| + |W_2 \cap R|.\]
Similarly to the case when $1 \in W_1$, $r \in W_1$ and $r^{-1} \in W_2$, one can show that $n \equiv 0 \mod 4$, which is a contradiction as $n$ is odd. The same contradiction can be reached if $r \in W_2$ and $r^{-1} \in W_1$. 

If $r \in W_2$ and $r^{-1} \in W_2$, then for any $r^i \in W_2 \cap T$,
\[2 + |W_2 \cap R| = |\Ga(1) \cap W_2| = |\Ga(r^i) \cap W_2| = |\{r^{i-1}, r^{i+1}\} \cap (W_2 \cap T)| + |W_2 \cap R|\]
and hence $\{r^{i-1}, r^{i+1}\} \subseteq W_2 \cap T$. Since $\Ga[\{1\} \cup T]$ is the cycle $(1, r, r^2, \ldots, r^{n-1}, 1)$, it follows that $W_2 \cap T = T$. So $T \subseteq W_2$ and $W_1 \cap T = \emptyset$. Since $\Ga[R]$ is a cycle, for each vertex $r^{i}s \in W_1 \cap R$, we have $b - 1 = |\Ga(r^i s) \cap W_1 \cap R| \le 2$, which contradicts the fact that $b \ge 6$. 

If $r \in W_1$ and $r^{-1} \in W_1$, then $|W_2 \cap R| = |\Ga(1) \cap W_2|$. Since $(1, r, r^2, \ldots, r^{n-1}, 1)$ is a cycle in $\Ga$ and both $r$ and $r^{-1}$ are in $W_1$, if $W_2 \cap T \neq \emptyset$, then there exists a vertex $r^i \in W_2 \cap T$ such that $|\Ga(r^i) \cap W_2 \cap T| = 1$, so $|\Ga(r^i) \cap W_2| = |W_2 \cap R|+1 \neq |\Ga(1) \cap W_2|$, which is a contradiction. Hence $W_2 \cap T = \emptyset$. That is, $T \subseteq W_1$, and hence all vertices in $T$ should have the same number of neighbours in $W_1$. However, $|\Ga(r) \cap W_1| = |\Ga(r^{-1}) \cap W_1| = |W_1 \cap R| + 1$ but $|\Ga(r^i) \cap W_1| = |W_1 \cap R| + 2$ for $r^i \in T \setminus \{r, r^{-1}\}$, a contradiction. 
\qed
\end{proof} 

\smallskip
\noindent \textbf{Acknowledgement}~~The third author was supported by a Discovery Project (DP250104965) of the Australian Research Council.


{\small

\begin{thebibliography}{99} 

\bibitem{BCGG2019}
R. A. Bailey, P. J. Cameron, A. L. Gavrilyuk and S. V. Goryainov, Equitable partitions of Latin-square graphs, {\em J. Combin. Des.} 27 (2019), no. 3, 142--160.

\bibitem{BCFFN24}
R. A. Bailey, P. J. Cameron, D. Ferreira, S. S. Ferreira and C. Nunes, Designs for half-diallel experiments with commutative orthogonal block structure, {\em J. Statist. Plann. Inference} 231 (2024), 106139.

\bibitem{BKMV2021}
E. A. Bespalov, D. S. Krotov, A. A. Matiushev and K. V. Vorob'ev, Perfect 2-colorings of Hamming graphs, {\em J Combin Des.} 29 (2021), no. 6, 367--396.

\bibitem{Big}
N. L. Biggs, Perfect codes in graphs, {\em J. Combin. Theory Ser. B} 15 (1973) 289--296.

\bibitem{Cardoso2019}
D. M. Cardoso, An overview of $(\kappa, \tau)$-regular sets and their applications, {\em Discrete Appl. Math.} 269 (2019) 2--10.

\bibitem{CWX2020}
J. Chen, Y. Wang and B. Xia, Characterization of subgroup perfect codes in Cayley graphs, {\em Discrete Math.} 343 (2020), no. 5, 111813.

\bibitem{DS03}
I. J. Dejter and O. Serra, Efficient dominating sets in Cayley graphs, 
\textit{Discrete Appl. Math.} 129 (2003) 319--328.

\bibitem{DS81} P. Diaconis and M. Shahshahani, Generating a random permutation with random transpositions, \emph{Z. Wahrscheinlichkeitstheor. Verw. Geb.} 57(2) (1981) 159--179.

\bibitem{E87}
G. Etienne, Perfect codes and regular partitions in graphs and groups, {\em European J. Combin.} 8 (1987), no. 2, 139--144.

\bibitem{Feit}
W. Feit, \emph{Characters of Finite Groups}, Benjamin, 1967. 

\bibitem{FHZ}
R. Feng, H. Huang and S. Zhou, Perfect codes in circulant graphs, {\em Discret. Math.} 340 (2017) 1522--1527.

\bibitem{F2007}
D. G. Fon-Der-Flaas, 
Perfect 2-colorings of a hypercube, {\em Siberian Math. J.} 48 (2007), no. 4, 740--745. (Translated from {\em Sibirsk. Mat. Zh.}, 48 (2007), no. 4, 923--930.)

\bibitem{GG2013}
A. L. Gavrilyuk and S. V. Goryainov, On perfect 2-colorings of Johnson graphs $J(v,3)$, {\em J. Combin. Des.} 21 (2013), no. 6, 232--252.

\bibitem{G1993}
C. Godsil, {\em Algebraic Combinatorics}, Chapman and Hall, New York, 1993.

\bibitem{GR}
C. Godsil and G. Royle, {\em Algebraic Graph Theory}, Springer-Verlag, New York, 2001. 

\bibitem{HHS98}
T. W. Haynes, S. T. Hedetniemi and P. Slater, {\em Fundamentals of Domination in Graphs}, Marcel Dekker, Inc., New York, 1998.

\bibitem{Heden1}
O. Heden, A survey of perfect codes, {\em Adv. Math. Commun.} 2 (2008) 223--247.

\bibitem{HXZ18}
H. Huang, B. Xia and S. Zhou, Perfect codes in Cayley graphs, {\em SIAM J. Discrete Math.} 32 (2018) 548--559.

\bibitem{K11}
D. S. Krotov, On weight distributions of perfect colorings and completely regular codes, {\em Des. Codes Cryptogr.} 61 (2011), no. 3, 315--329.

\bibitem{KP25}
D. S. Krotov and V. N. Potapov, Completely regular codes and equitable partitions, in: {\em Completely Regular Codes in Distance Regular Graphs}, 1--84, Monogr. Res. Notes Math., CRC Press, Boca Raton, FL, 2025.

\bibitem{Lee01}
J. Lee, Independent perfect domination sets in Cayley graphs, {\em J. Graph Theory} 37 (2001) 213--219.

\bibitem{LZ2022}
X. Liu, S. Zhou, Eigenvalues of Cayley graphs, \emph{Electron. J. Combin.} 29 (2) (2022) Paper No. 2.9, 164 pp.

\bibitem{MWWZ2019}
X. Ma, G. L. Walls, K. Wang and S. Zhou, Subgroup perfect codes in Cayley graphs, {\em SIAM J. Discrete Math.} 34 (2020), no.3, 1909--1921.

\bibitem{M03}
A. Meyerowitz, Cycle-balance conditions for distance-regular graphs, {\em Discrete Math.} 264 (2003), no. 1-3, 149--165.

\bibitem{MV2020}
I. Mogilnykh and A. Valyuzhenich, Equitable 2-partitions of the Hamming graphs with the second eigenvalue, {\em Discrete Math.} 343 (2020), no. 11, 112039.

\bibitem{Neu92}
A. Neumaier, Completely regular codes, {\em Discrete Math.} 106/107 (1992) 353--360.

\bibitem{RT1966}
O. Rothaus and J. G. Thompson, A combinatorial problem in the symmetric group, {\em Pacific J. Math.} 18 (1966) 175--178.

\bibitem{SS2009}
S. Szab\'{o} and A. Sands, {\em Factoring Groups into Subsets}, CRC Press, Boca Raton, FL, 2009.

\bibitem{WXZ23}
Y. Wang, B. Xia and S. Zhou, Regular sets in Cayley graphs, \emph{J. Algebraic Combin.} 57 (2023), no. 2, 547--558.

\bibitem{WXZ24}
X. Wang, S-J. Xu and S. Zhou, On regular sets in Cayley graphs, \emph{J. Algebraic Combin.} 59 (2024), no. 3, 735--759.

\bibitem{Va75}
J. H. van Lint, A survey of perfect codes, {\em Rocky Mountain J. Math.} 5 (1975) 199--224.

\bibitem{ZZ2021} 
J. Zhang and S. Zhou, On subgroup perfect codes in Cayley graphs, {\em European J. Combin.} 91 (2021) 103228. (Corrigenda:  \emph{European J. Combin.} 101 (2022) 103461)

\bibitem{Zhou2016}
S. Zhou, Total perfect codes in Cayley graphs, {\em Des. Codes Cryptogr.} 81 (2016), no. 3, 489--504.

\bibitem{Z15}
S. Zhou, Cyclotomic graphs and perfect codes, {\em J. Pure Appl. Algebra} 223 (2019) 931--947.

\bibitem{Z88} 
P. H. Zieschang, Cayley graphs of finite groups, \emph{J. Algebra} 118(2) (1988) 447--454.
\end{thebibliography}
}
\end{document}